\documentclass{article}

\usepackage[
  left=1.25in,
  right=1.25in,
  top=1.2in,
  bottom=1.2in
]{geometry}

\usepackage{amsmath, amssymb, amsthm}
\usepackage{mathtools}
\usepackage{xcolor}
\usepackage{hyperref}
\usepackage{enumitem}
\usepackage[style=numeric-comp,hyperref=true,doi=false,url=false,
            isbn=false,giveninits=true,sorting=none,
            block=none,backend=bibtex,maxnames=99]{biblatex}
\newcommand{\F}{\mathbb{F}}

\newcommand{\coinf}{\mathrm{C}_{\mathsf{co}}^\infty}
\newcommand{\N}{\mathbb{N}}
\newcommand{\R}{\mathbb{R}}
\newcommand{\C}{\mathbb{C}}
\newcommand{\bfo}{\mathbf{1}}
\newcommand{\tr}{\operatorname{tr}}
\newcommand{\frakg}{\mathfrak{g}}
\newcommand{\frakh}{\mathfrak{h}}

\newcommand{\xalg}{\mathfrak X_{\mathrm{alg}}}
\newcommand{\xc}[1]{\vspace{.1cm}
\noindent {\em #1} }

\newtheorem{theorem}{Theorem}
\newtheorem{proposition}[theorem]{Proposition}

\newtheorem{lemma}{Lemma}
\theoremstyle{definition}
\newtheorem{remark}{Remark}

\theoremstyle{plain}
\newtheorem{mainthm}{Theorem}

\newtheorem{maincor}{Corollary}

\newtheorem{controlcor}{Corollary}

\title{Algebraic Vector Fields of Connected Semi-simple Matrix Groups are $3$-generated}

\author{Yiyang Jiang$^*$ \quad and \quad Xudong Chen\footnote{Y. Jiang and X. Chen are with the Department of Electrical \& Systems Engineering, Washington University, St.~Louis, MO 63130, USA. Emails: {\tt\small \char123 j.yiyang, cxudong\char125 @wustl.edu}.
}}

\date{}

\begin{document}

\maketitle

{\renewcommand{\thefootnote}{}\footnotetext{\emph{2020 Mathematics Subject Classification.} Primary 17B66; Secondary 22E46, 32M17, 58D05, 93B05.}}

\begin{abstract}
We show that the Lie algebra of algebraic vector fields for any connected, semi-simple matrix group over $\F$, for $\F = \R, \C$,  can be generated by three $\F$-complete algebraic vector fields.
\end{abstract}

\section{Introduction and main results}\label{sec:introduction}

Let $\F$ be either $\R$ or $\C$.
Let $G$ be a connected, semi-simple matrix group over $\F$, by which we mean a connected, closed subgroup of the general linear group $\mathrm{GL}(\mathbb{V})$ for some finite-dimensional vector space $\mathbb{V}$ over $\F$, whose Lie algebra $\frakg$ over~$\F$ is semi-simple; for $\F = \C$, we further require that $G$ be a complex Lie group. 
The class of real Lie groups considered in the paper includes, for example, all the identity components of the classical Lie groups given in~\cite[Section~I.17]{Knapp}. 
We also refer the reader to~\cite[Chapter~5]{OnishchikV} for a comprehensive study of these real Lie groups. 
In contrast to the real case, every connected, complex semi-simple Lie group is a matrix group (see, e.g.,~\cite[Chapter~4]{OnishchikV}).

Let $\Phi$ be the space over $\F$ spanned by the matrix coefficients associated with the representation of $G$ on~$\mathbb{V}$, and $\mathcal{A}(\Phi)$ be the unital algebra generated by $\Phi$.
A vector field $V$ on $G$ is viewed as a map $V:G\to\operatorname{End}(\mathbb{V})$.
Let $\mathfrak{X}(G)$ be the space of smooth (resp., holomorphic) vector fields on~$G$ for $\F=\R$ (resp., $\F=\C$). 
The vector field $V$ is said to be \emph{algebraic} if, in a basis of $\mathbb{V}$, all matrix entries of $V$ belong to $\mathcal{A}(\Phi)$. Let $\xalg(G)$ be the space of algebraic vector fields.
We say that $V$ is {\it $\F$-complete} (or simply {\it complete} if there is no ambiguity) if its flow $e^{tV}: G\to G$ is well defined for all $t\in \F$.
Given a set $S\subseteq \mathfrak{X}(G)$, let $\operatorname{Lie}(S)$ be the Lie algebra over $\F$ generated by~$S$.

The main result of the paper is the following:

\begin{mainthm}\label{thm:main}
For any connected, semi-simple matrix group $G$ over $\F$,
there exist three $\F$-complete algebraic vector fields $W_1$, $W_2$, and $W_3$ such that
$\xalg(G) = \operatorname{Lie}\{W_1,W_2,W_3\}$.
\end{mainthm}

We equip $\mathfrak{X}(G)$ with the compact-open $\mathrm{C}^\infty$-topology (or simply $\coinf$-topology). 
For $\F = \C$, the $\coinf$-topology is nothing but the compact-open topology by Cauchy's estimate. It is known~\cite[Theorem~3.4]{TV06} that $\xalg(G)$ is dense in $\mathfrak{X}(G)$. 
For $\F = \R$, we will show in Lemma~\ref{lem:smooth-Vfieldapproximate}, Section~\ref{sec:consequences} that $\xalg(G)$ is dense in $\mathfrak{X}(G)$.

The so-called {\it density property} is introduced in~\cite{Varolin01,TV06} for complex manifolds~$M$, which requires that every holomorphic vector field on $M$ can be approximated, uniformly on compact sets, by Lie combinations of complete holomorphic vector fields.
It has been shown in~\cite{TV00} that every complex semi-simple Lie group has the density property.
Borrowing this terminology, we say that a real/complex $M$ has the {\it strong density property} if there is a {\it finite} set $S$ of $\F$-complete vector fields such that $\operatorname{Lie}(S)$ is dense in $\mathfrak{X}(M)$.
It follows directly from Theorem~\ref{thm:main} that any connected, semi-simple matrix group $G$ over~$\F$ has the strong density property.

The fact that a complex~$G$ has the strong density property can also be obtained by combining the following two results:
On one hand, it has been shown~\cite[Theorem~3]{KK08} that every algebraic vector field on~$G$ is a Lie combination of complete algebraic vector fields. On the other hand, it has recently been shown~\cite[Theorem~1.1]{Bouaziz26} 
that $\xalg(M)$, for any smooth complex affine variety~$M$, can be generated by a finite number of algebraic vector fields.

In terms of the number of $\C$-complete vector fields needed for generating $\xalg(M)$, the following is known for complex manifolds~$M$: For the complex special linear group $\mathrm{SL}_2(\C)$, it has been shown in~\cite[Theorem~14]{AndristSLtwo}  that $\xalg(\mathrm{SL}_2(\C))$ is generated by $4$ $\C$-complete vector fields (or simply $4$-generated).
The result has recently been improved~\cite[Proposition~3.6]{AndristRegeta26} that $\xalg(\mathrm{SL}_2(\C))$ is $3$-generated. 
For~$\C^d$ with $d \geq 2$, it has been shown in~\cite[Corollary~2]{BeldievPogudin} that $\xalg(\C^d)$ is $2$-generated, which builds upon the results of~\cite{AndristCn}. 
For $M$ the Danielewski surface $xy=z^2$, $\xalg(M)$ is $5$-generated~\cite{AndristSLtwo}. Further, for any Danielewski surface $M$ of the form $xy=p(z)$, where $p(z)$ is a polynomial with simple zeros, $\xalg(M)$ is $6$-generated~\cite{AndristDanielewski}.

For real manifolds, it has been shown in~\cite[Section~6]{AS22} that $\R^d$, the $d$-torus $\mathbb{T}^d$, and the $2$-sphere $S^2$ have the strong density property, with $2d$, $4d$, and $7$ $\R$-complete generators, respectively.
For $S^2$, we have recently shown in~\cite{CofDS2} that $\xalg(S^2)$ is $2$-generated.

We now describe a relevant consequence of Theorem~\ref{thm:main}.
Let $\operatorname{Diff}(G)$ be the group of smooth (resp., biholomorphic) diffeomorphisms of $G$ for $\F=\R$ (resp., $\F = \C$).  
We equip $\operatorname{Diff}(G)$ with the $\coinf$-topology. 
Again, for $\F=\C$, this topology coincides with the compact-open topology. 
Let $\operatorname{Diff}_0(G)$ be the subgroup of $\operatorname{Diff}(G)$, comprising elements that are smoothly isotopic to the identity map~$\operatorname{id}_G$. 
Let $W_1$, $W_2$, and $W_3$ be given as in the statement of Theorem~\ref{thm:main}, and $\mathcal{W}$ be the subgroup of $\operatorname{Diff}_0(G)$ generated by $e^{tW_i}$, for $t\in \F$ and $i = 1,2,3$.

The following result is a corollary of Theorem~\ref{thm:main}:

\begin{maincor}\label{cor:B}
The subgroup $\mathcal{W}$ is dense in $\operatorname{Diff}_0(G)$  with respect to the $\coinf$-topology.
\end{maincor}

It has been shown in~\cite{AC09} that if $V_1,\ldots, V_m$ are smooth vector fields on a real, compact manifold $M$ and if the group generated by $e^{tV_i}$, for $t\in \R$ and $i = 1,\ldots, m$, acts transitively on $M$, then for any $f\in\operatorname{Diff}_0(M)$, there exist an $N\in \N$, a sequence $i_1,\ldots, i_N\in \{1,\ldots, m\}$, and $a_1,\ldots, a_N\in \mathrm{C}^\infty(M)$ such that $f = e^{a_N V_{i_N}}\circ \cdots \circ e^{a_1 V_{i_1}}$.
However, their result does not imply ours as the coefficients $a_1,\ldots, a_N$ in our setting can only be constants.

We further relate the above corollary to geometric control theory.
Let $G$ be a connected, semi-simple matrix group over~$\R$.  
For a nonnegative integer~$k$, let $J^k(G,G)$ be the manifold of $k$-jets of smooth maps from $G$ to $G$. We denote by $j_x^kf$ the $k$-jet of $f\in \mathrm{C}^\infty(G, G)$ at $x\in G$.
We embed $J^k(G,G)$ into a Euclidean space $\R^{N_k}$. 
For smooth $G$-valued maps $f,g$ defined on an open neighborhood of a compact set $K$ in $G$, let
$d_{k,K}(f,g):=\max_{x\in K}\left\|j_x^k f-j_x^k g\right\|$, 
where $\|\cdot\|$ is the standard Euclidean norm.
Now, consider the following control-linear system on the group $G$:
\begin{equation}\label{eq:controlsystem}
\dot x(t) = \sum_{i = 1}^3 u_i(t) W_i(x(t)),
\end{equation}
where each $u_i: [0,\infty)\to \R$, for $i = 1,2,3$, is a locally integrable function, viewed as a scalar control input. For ease of notation, let $u:=(u_1,u_2, u_3)$.
System~\eqref{eq:controlsystem} is said to be {\it approximately controllable} on $\operatorname{Diff}_0(G)$ in the $\coinf$-sense over the time interval $[0,T]$ for a given $T > 0$, if for any $f\in \operatorname{Diff}_0(G)$, any compact set $K\subseteq G$, any nonnegative integer~$k$, and any error tolerance $\epsilon > 0$, there exists a $u\in \mathrm{L}^1([0,T],\R^3)$ such that the flow $g_t$ generated by the time-varying vector field $U_t(x):=\sum_{i = 1}^3 u_i(t) W_i(x)$ is well defined on an open neighborhood of $K$ in $G$ for all $t\in [0,T]$ and, moreover, $d_{k,K}(f, g_T) < \epsilon$.
Since system~\eqref{eq:controlsystem} is control-linear,  approximate controllability over $[0,T]$ for some $T > 0$ implies approximate
controllability over $[0,T]$ for {\it any} $T > 0$.

The following result is then an immediate consequence of Corollary~\ref{cor:B}:

\begin{controlcor}\label{cor:control}
System~\eqref{eq:controlsystem} is approximately controllable on $\operatorname{Diff}_0(G)$ in the $\coinf$-sense.
\end{controlcor}

For relevant work about controllability of diffeomorphisms, we mention~\cite{AS22} in which the authors have considered the class of control-linear systems on real, Riemannian manifolds~$M$. In particular, they have introduced the so-called ``strong Lie algebra approximating property'' (Assumption~3), which, roughly speaking, serves as the same purpose of Theorem~\ref{thm:main} (although the two conditions are significantly different).  Assuming that the control vector fields and their covariant derivatives of each other are uniformly bounded and that the aforementioned property is satisfied, the authors have shown that the control-linear system is approximately controllable on $\operatorname{Diff}_0(M)$ with respect to the compact-open topology.

The remainder of the paper is organized as follows: We gather key notations used throughout the paper at the end of this section. 
In Section~\ref{sec:degree-growth}, we investigate the Lie algebra generated by $\Phi\otimes\frakg_L$ (or simply $\Phi\frakg_L$), 
where $\frakg_L$ is the space of left-invariant vector fields, and show that $\operatorname{Lie}(\Phi\frakg_L)= \mathcal{A}(\Phi)\frakg_L$.  
In Section~\ref{sec:shared}, we exhibit three complete vector fields $W_1$, $W_2$, and $W_3$ such that $\Phi\frakg_L \subseteq \operatorname{Lie}\{W_1,W_2,W_3\}$. 
In Section~\ref{sec:proof-main-result}, we prove Theorem~\ref{thm:main} and Corollary~\ref{cor:B}.

\xc{Notation:} We use $\mathbb{N}$ to denote the set of positive integers and $\mathbb{N}_0 := \N \cup\{0\}$. 
Let~$\frakg$ be the Lie algebra associated with~$G$. 
Let~$\frakg_L$ (resp., $\frakg_R$) be the Lie algebra of left-invariant (resp., right-invariant) vector fields on~$G$.
We identify~$\frakg_L$ with~$\frakg$.
For a given $X\in \frakg$, let $L_X$ (resp., $R_X$) be the corresponding left-invariant (resp., right-invariant) vector field, i.e., $L_X(x)=xX$ and $R_X(x)=Xx$ for any $x\in G$.

Let $\mathfrak{s}$ be an arbitrary semi-simple Lie algebra and $\mathcal{U}(\mathfrak{s})$ be the universal enveloping algebra of~$\mathfrak{s}$. Given an arbitrary representation $\mathbb{V}$ of $\mathfrak{s}$, we use $Xv\in \mathbb{V}$ to denote the action of $X\in \mathcal{U}(\mathfrak{s})$ 
on the vector~$v\in \mathbb{V}$. 
For a subset $S\subseteq \mathbb{V}$, we use $\mathcal{U}(\mathfrak{s})S$ to denote the space spanned by $Xv$ for $X\in \mathcal{U}(\mathfrak{s})$ and $v\in S$, which is the smallest $\mathfrak{s}$-invariant subspace of $\mathbb{V}$ that contains~$S$.

For now, let $\frakg$ be a complex semi-simple Lie algebra. 
Let $\mathfrak{h}$ be a Cartan subalgebra of $\frakg$, $\Delta$ be the set of roots associated with $(\frakg, \frakh)$, $\frakg = \frakh \oplus \bigoplus_{\alpha \in \Delta} \frakg_\alpha$ be the root space decomposition, and $\Delta^+$ be a positive system.
Let $B(X, Y):= \operatorname{tr}(\operatorname{ad}_X\operatorname{ad}_Y)$ be the Killing form.
For each $\alpha\in \Delta$, let $h_\alpha\in \frakh$ be defined such that $\alpha(H) = B(h_\alpha, H)$ for any $H\in \frakh$.
Let $\langle \alpha, \beta \rangle:= B(h_\alpha, h_\beta)$  and extend it bilinearly to $\frakh^*\times \frakh^*$. 
Note, in particular, that $\langle \cdot, \cdot \rangle$ is an inner-product on the $\R$-span of $\Delta$.
For each nonzero $\alpha\in \frakh^*$ and for all $\beta\in \frakh^*$, we follow the notation of~\cite{Knapp} and define $$(\beta,\alpha):= 2\langle \beta, \alpha\rangle/\|\alpha\|^2.$$
Further, let $H_\alpha \in \frakh$ be defined such that $\beta(H_{\alpha}) = (\beta, \alpha)$.
In particular, $\alpha(H_\alpha) = 2$; for $\alpha\in \Delta$, $H_\alpha$ is the coroot of~$\alpha$.
Let $\Lambda_R$ be the root lattice, i.e., the $\mathbb{Z}$-span of $\Delta$, and $\Lambda_W$ be the weight lattice, i.e., the set of $\lambda\in \frakh^*$ with $(\lambda,\alpha)\in \mathbb{Z}$ for all $\alpha\in \Delta$.
A weight $\lambda\in \Lambda_W$ is {\it minuscule} if $(\lambda,\alpha)\in \{0,1\}$ for all $\alpha\in \Delta^+$.

\section{Lie algebra generated by degree-one vector fields}\label{sec:degree-growth}
In this section, let $G$ be an arbitrary connected, semi-simple Lie group over $\F$ (not necessarily a matrix group). Let $\Phi$ be an arbitrary finite-dimensional space of smooth (resp., holomorphic) functions on $G$ for $\F = \R$ (resp., $\F = \C$), closed under left-invariant vector fields.
For each $n\in \N$, let $\mathcal{A}_n(\Phi)$ be the space spanned by $\phi_1\cdots \phi_n$ with $\phi_i\in \Phi$, and $\mathcal{A}_0(\Phi) := \F\bfo$, where $\bfo$ is the constant function $\bfo(x) = 1$ for all $x\in G$.
Let $\mathcal{L}_n$ be the space of right-normed Lie products $[V_1,\cdots,[V_{n-1}, V_n]]$ of depth~$n$ with $V_i\in \Phi \frakg_L$. Note that by the Jacobi identity, any Lie product is a linear combination of right-normed Lie products.

The main result of the section is the following:

\begin{theorem}
\label{thm:liealgebra}
If $\Phi$ does not contain any
nonzero constant function, then for any $n \in \N$,
$$
\mathcal{L}_n=\mathcal{A}_n(\Phi)\mathfrak g_L.
$$
\end{theorem}

For the case $\F = \R$, we will prove the theorem by establishing its complexified version, i.e., $\mathcal{L}^\C_n = \mathcal{A}_n(\Phi)^\C \frakg_L^\C$. 
For ease of notation, we will omit their super-index for the remainder of the section. All spaces below are thus over~$\C$ unless otherwise stated.
Also, we will use $X\in \frakg$ as a shorthand notation for $L_X$, so for any $\phi\in \Phi$ we write $X\phi$ instead of $L_X \phi$.

The proof will be carried out by induction on $n$. For the base case $n = 1$, the result holds trivially. For the inductive step, we need to show that
\begin{equation*}
[\Phi\frakg_L, \mathcal{A}_{n-1}(\Phi) \frakg_L] = \mathcal{A}_n(\Phi) \frakg_L.
\end{equation*}

Note that for any $\phi\in \Phi$, any $\zeta := \phi_1\cdots \phi_{n-1}\in \mathcal{A}_{n-1}(\Phi)$, and any $X, Y\in \frakg_L$,
\begin{equation}\label{eq:liebracket}
[\phi X, \zeta Y ] = \phi\zeta [X, Y] + (X\zeta) \phi Y - (Y\phi) \zeta X.
\end{equation}
Since  $\mathcal{A}_n(\Phi)$ is closed under $\frakg_L$ for all $n \in\N$, the right hand side of~\eqref{eq:liebracket} belongs to $\mathcal{A}_n(\Phi)\frakg_L$, so $[\Phi\frakg_L, \mathcal{A}_{n-1}(\Phi) \frakg_L] \subseteq \mathcal{A}_n(\Phi) \frakg_L$.
The main effort of the proof is in establishing the reverse inclusion:
\begin{equation}\label{eq:converseinclusion}
[\Phi\frakg_L, \mathcal{A}_{n-1}(\Phi) \frakg_L] \supseteq \mathcal{A}_n(\Phi) \frakg_L.
\end{equation}

As a consequence of the Chevalley basis, there exist elements $X_\alpha\in \frakg_\alpha$, for $\alpha\in \Delta$, such that the following commutation relations hold:
$$
\begin{aligned}
\relax
[X_\alpha, X_{-\alpha}] & = H_\alpha, \quad \mbox{for any } \alpha \in \Delta, \\
[H_\alpha, X_\beta] & = (\beta,\alpha) X_\beta,  \quad \mbox{for any } \alpha, \beta\in \Delta,  \\
[X_\alpha, X_\beta] & = c_{\alpha, \beta} X_{\alpha + \beta}, \quad \mbox{for nonproportional } \alpha \mbox{ and } \beta,
\end{aligned}
$$
where $c_{\alpha,\beta}^2 = (r + 1)^2$, with $r$ the largest integer such that $(\beta - r\alpha)$ is a root (by convention, $X_{\alpha + \beta} = 0$ if $\alpha + \beta \notin\Delta \cup \{0\}$).

Throughout this section, let $\Lambda$ be the set of weights associated with the representation $\Phi$ of~$\frakg$, and $
\Phi = \oplus_{\lambda \in\Lambda}\Phi_\lambda$
be the weight space decomposition.  
Note that $(\lambda,\alpha)\in\mathbb{Z}$ for all $\lambda\in\Lambda$ and $\alpha\in\Delta$.  
It is clear that $\mathcal{A}_n(\Phi)\frakg_L$ is spanned by elements $\xi H_{\alpha}, \xi X_{\alpha}$, where
 $\alpha\in \Delta$ and $$\xi = \prod_{i = 1}^n\phi_i, \quad \mbox{ with } \phi_i\in \Phi_{\lambda_i} \mbox{ for some } \lambda_i\in \Lambda.$$
For convenience, for each $i = 1,\ldots, n$, let
$\xi_{-i}:= \prod_{j = 1, j\neq i}^n \phi_j$. Also, let
\begin{equation}\label{eq:defsigmaanddelta}
\lambda:= \frac{1}{2}\sum_{i = 1}^{n} \lambda_i.
\end{equation}

We show below that $\xi H_\alpha, \xi X_\alpha\in [\Phi\frakg_L, \mathcal{A}_{n-1}(\Phi)\frakg_L]$. 
We start with the following proposition:

\begin{proposition}\label{prop:case1}
    If there exists some $i$ such that $\lambda_i\neq \lambda$, then $\xi H\in \mathcal{L}_n $ for any $H\in \frakh$.
\end{proposition}

\begin{proof}
Let $\delta_i:= 2(\lambda - \lambda_i) \neq 0$.
It directly follows from computation that
$$
[\phi_i H, \xi_{-i} H]  = (H \xi_{-i}) \phi_i H - (H\phi_i) \xi_{-i} H
= \delta_i(H) \xi H.
$$
If $\delta_i(H)\neq 0$, then $\xi H\in \mathcal{L}_n$.
Otherwise, let $H':= H + H_{\delta_i}$. Since $\delta_i(H') = \delta_i(H_{\delta_i}) = 2$, both $\xi H'$ and $\xi H_{\delta_i}$ belong to $\mathcal{L}_n$ and hence,
$\xi H = \xi H' -  \xi H_{\delta_i} \in \mathcal{L}_n$.
\end{proof}

We now assume that $\lambda_i = \lambda$ for all $i = 1,\ldots, n$ and show that $\xi H\in \mathcal{L}_n$. It follows from~\eqref{eq:defsigmaanddelta} that $(n - 2)\lambda = 0$.
Consider the following two cases:
\begin{description}
    \item[\it Case 1:] $n = 2$ and $\lambda \neq 0$;
    \item[\it Case 2:] $\lambda = 0$.
\end{description}

We first deal with Case~1. For any given $\alpha\in \Delta$, let $\mathfrak{s}_\alpha$ be the Lie subalgebra of $\frakg$ spanned by $\{H_\alpha, X_\alpha, X_{-\alpha}\}$, which is isomorphic to $\mathfrak{sl}_2(\C)$. We need the following lemma:

\begin{lemma}\label{lem:sl2rep}
For any $\phi \in \Phi_\lambda$ and for any $\alpha\in \Delta$ with $(\lambda, \alpha) \geq 0$, there exists a decomposition  $\phi = \phi' + \phi''$, with $\phi',\phi''\in \Phi_\lambda$, such that the following hold
\begin{enumerate}
\item There exists an element $\psi\in \Phi_{\lambda - \alpha}$ such that $X_\alpha  \psi = \phi'$.
\item  $\mathfrak{s}_\alpha \phi'' = \{0\}$.
\end{enumerate}
\end{lemma}

\begin{proof}
Let $\Phi'$ be the smallest subspace of $\Phi$ that contains $\phi$ and is closed under $\mathfrak{s}_\alpha$, so $\Phi'$ is a representation of $\mathfrak{s}_\alpha$.
Decompose $\Phi'= \oplus_{\ell}\Phi'_{\ell}$ such that each $ \Phi'_{\ell}$ is irreducible.
Correspondingly, we write $\phi = \sum_{\ell} \phi_{\ell}$ where  $\phi_{\ell}\in \Phi'_{\ell}$.
Note that each $\phi_\ell$  necessarily belongs to $\Phi_\lambda$.
Let  $m_\ell\in \N_0$ be the highest weight of~$\Phi'_{\ell}$, and $k:=(\lambda, \alpha) \geq 0$.
Then, $m_\ell \geq k$ and, moreover, $m_\ell$ and $k$ share the same parity.
Let
$\phi':= \sum_{\ell: m_\ell > 0} \phi_\ell$ and
$\phi'':= \sum_{\ell: m_\ell = 0} \phi_\ell$.
By construction, $\mathfrak{s}_\alpha \phi'' = 0$.
For $\phi'$, note that $${X_\alpha}{X_{-\alpha}}\phi_{\ell} = \frac{1}{4}( m_\ell - k + 2)(m_\ell + k)\phi_{\ell}.$$
Thus, by setting $$\psi:= \sum_{\ell:m_\ell > 0} \frac{4}{(m_\ell - k + 2)(m_\ell + k)} {X_{-\alpha}} \phi_{\ell},$$
we obtain ${X_\alpha}\psi = \phi'$.
\end{proof}

With the lemma above, we now establish the following proposition:

\begin{proposition}\label{prop:case2}
    Let $\phi_1, \phi_2\in \Phi_\lambda$, with $\lambda\neq 0$.
    Then, $\phi_1\phi_2 H_\alpha \in \mathcal{L}_2$ for any $\alpha\in \Delta$.
\end{proposition}

\begin{proof}
We consider the following two sub-cases:

\xc{Sub-case 1.1: $(\lambda, \alpha) \neq 0$.} Without loss of generality, we assume that $(\lambda, \alpha) > 0$. Then, $\mu:=\lambda - \alpha$ is a weight.
We appeal to Lemma~\ref{lem:sl2rep} to obtain the decomposition $\phi_1 = \phi_1' + \phi_1''$. Note that $H_\alpha \phi_1'' = (\lambda, \alpha) \phi_1''$.
Since $(\lambda, \alpha)\neq 0$, we must have that $\phi_1'' = 0$ and hence, $\phi_1' = \phi_1$.
Let $\psi_1\in \Phi_\mu$ be such that $X_\alpha \psi_1 = \phi_1$. Then, by~\eqref{eq:liebracket},
\begin{equation}\label{eq:case2-1}
\begin{aligned}\relax
[\phi_2 H_\alpha, \psi_1  X_\alpha] & = \psi_1\phi_2[H_\alpha, X_\alpha] + (\mu,\alpha)\psi_1 \phi_2 X_\alpha - \psi_1 (X_\alpha \phi_2) H_\alpha \\
& = (\lambda, \alpha) \psi_1 \phi_2 X_\alpha - \psi_1 (X_\alpha \phi_2) H_\alpha \in \mathcal{L}_2.
\end{aligned}
\end{equation}
Note that $X_\alpha \phi_2\in \Phi_\nu$, where $\nu:= \lambda + \alpha$.
Since $\nu - \mu = 2\alpha \neq 0$, it follows from Proposition~\ref{prop:case1} that $\psi_1 \psi_2 H_\alpha\in \mathcal{L}_2$ for any $\psi_2\in \Phi_\nu$. Using~\eqref{eq:case2-1} and the fact that $(\lambda, \alpha)\neq 0$, we have that $\psi_1 \phi_2 X_\alpha \in \mathcal{L}_2$. Using~\eqref{eq:liebracket} again, we obtain that
\begin{equation*}
\begin{aligned}\relax
[\psi_1 H_\alpha, \phi_2  X_\alpha] & = \psi_1\phi_2[H_\alpha, X_\alpha] + (\lambda, \alpha) \psi_1\phi_2 X_\alpha - (X_\alpha \psi_1) \phi_2 H_\alpha \\
& = (\nu, \alpha) \psi_1\phi_2 X_\alpha - \phi_1\phi_2 H_\alpha \in \mathcal{L}_2.
\end{aligned}
\end{equation*}
Since $\psi_1\phi_2 X_\alpha\in \mathcal{L}_2$, we have that $\phi_1\phi_2 H_\alpha\in \mathcal{L}_2$.

\xc{Sub-case 1.2: $(\lambda, \alpha) = 0$.}
We use Lemma~\ref{lem:sl2rep} to decompose $\phi_i = \phi'_i + \phi''_i$, for $i= 1,2$. In case $\phi'_i$ is nonzero, we let $\psi_i\in \Phi_{\lambda- \alpha}$ be such that $X_\alpha\psi_i = \phi'_i$.
We then write
\begin{equation}\label{eq:case2-3}
\phi_1 \phi_2 H_\alpha = \phi'_1\phi'_2 H_\alpha + \phi'_1\phi''_2 H_\alpha + \phi''_1 \phi'_2 H_\alpha + \phi''_1\phi''_2 H_\alpha.
\end{equation}
We show below that each term on the right hand side of~\eqref{eq:case2-3} belongs to $\mathcal{L}_2$.
For the first term $\phi'_1 \phi'_2 H_\alpha$, we let $H':= H_\alpha + H_\lambda$ and note that $\lambda(H')- \alpha(H') = 0$.
Then,
$$
\begin{aligned}\relax
[\phi'_2 H', \psi_1 X_\alpha ] & = \psi_1 \phi'_2 \alpha(H') X_\alpha + (\lambda - \alpha)(H') \phi'_2\psi_1 X_\alpha - \psi_1 (X_\alpha \phi'_2) H' \\
& = 2 \psi_1 \phi'_2 X_\alpha - \psi_1 (X_\alpha \phi'_2) H' \in \mathcal{L}_2, \\
[\psi_1 H_\alpha, \phi'_2 X_\alpha] & = 2\psi_1 \phi'_2 X_\alpha - \phi'_1\phi'_2 H_\alpha \in \mathcal{L}_2.
\end{aligned}
$$
Using the same arguments as in Sub-case~1.1, $\psi_1 (X_\alpha \phi'_2) H'\in \mathcal{L}_2$ which implies that $\psi_1\phi'_2 X_\alpha\in \mathcal{L}_2$ and hence, $\phi'_1\phi'_2 H_\alpha\in \mathcal{L}_2$.
For the second term $\phi'_1\phi''_2 H_\alpha$ of~\eqref{eq:case2-3}, we similarly have that
$$
\begin{aligned}\relax
[\phi''_2 H', \psi_1 X_\alpha ] & = 2\psi_1 \phi''_2 X_\alpha \in \mathcal{L}_2,  \\
[\psi_1 H_\alpha, \phi''_2 X_\alpha] & = 2\psi_1 \phi''_2 X_\alpha - \phi'_1\phi''_2 H_\alpha \in \mathcal{L}_2,
\end{aligned}
$$
so $\phi'_1\phi''_2 H_\alpha\in \mathcal{L}_2$. By the same arguments (and by symmetry), the third term $\phi''_1\phi'_2 H_\alpha$ also belongs to $\mathcal{L}_2$.
Finally, for the last term $\phi''_1\phi''_2 H_\alpha$, we have that
$$
[\phi''_1 X_\alpha, \phi''_2 X_{-\alpha}] = \phi''_1\phi''_2 H_\alpha \in \mathcal{L}_2.
$$
This completes the proof.
\end{proof}

We now deal with Case 2, i.e., $\lambda_i = 0$ for all $i = 1,\ldots, n$.
By the hypothesis of Theorem~\ref{thm:liealgebra}, $\Phi$ does not contain any nonzero constant function, so $\Phi$ does not contain a trivial representation of $\frakg$.
It follows from the highest weight theorem that the weight space $\Phi_0$ can be expressed as
$\Phi_0 = \sum_{\alpha\in \Delta^+} X_{\alpha}\Phi_{-\alpha}$.
By multi-linearity, we can assume, without loss of generality,
that $\phi_1,\ldots, \phi_n$ satisfy the condition that there exist $\alpha_i\in \Delta^+$ and $\psi_i \in \Phi_{-\alpha_i}$, for $i = 1,\ldots, n$, such that
\begin{equation}\label{eq:xalphapsi}
X_{\alpha_i}\psi_i  = \phi_i.
\end{equation}
Let $\Phi_i$ be the smallest subspace of $\Phi$ that contains $\psi_i$ and is closed under $\mathfrak{s}_{\alpha_i}$.
Decompose $\Phi_i = \oplus_\ell\Phi_{i,\ell}$, where each $\Phi_{i,\ell}$ is irreducible under $\mathfrak{s}_{\alpha_i}$.
Let $m_{i,\ell}$ be the highest weight of $\Phi_{i,\ell}$.
Since $0$ is a weight of $\Phi_{i,\ell}$, $m_{i,\ell}$ is an even integer. Moreover, it follows from~\eqref{eq:xalphapsi} that $m_{i,\ell}$ must be strictly positive because otherwise no such $\psi_i$ would exist for the relation to hold.
For ease of presentation but without loss of generality, we can assume that the representation $\Phi_i$  is itself irreducible.

We now establish the following result:

\begin{proposition}\label{prop:case3}
    Suppose that  each $\phi_i$ is a zero-weight vector of the irreducible representation~$\Phi_i$ of~$\mathfrak{s}_{\alpha_i}$ with highest weight $m_i > 0$; then, $\xi H\in \mathcal{L}_n$ for all $H\in \mathfrak{h}$.
\end{proposition}

\begin{proof}
Let $\theta_i,\psi_i\in \Phi_{i}$ be such that $X_{-\alpha_i} \theta_i = X_{\alpha_i}\psi_i = \phi_i$.
We first show that $\xi H_{\alpha_i} \in \mathcal{L}_n$ for all $i = 1,\ldots, n$.
We consider two sub-cases:

\xc{Sub-case 2.1: There exist $i,j\in \{1,\ldots, n\}$ such that $\alpha_i \neq \alpha_j$.}  By computation,
\begin{equation*}
\begin{aligned}\relax
[\psi_i H_{\alpha_i}, \xi_{-i} X_{\alpha_i}] & = 2\psi_i \xi_{-i} X_{\alpha_i} - \xi H_{\alpha_i}, \\
[\psi_j H_{\alpha_i}, \xi_{-j} X_{\alpha_j}] & = (\alpha_j, \alpha_i) \psi_j\xi_{-j} X_{\alpha_j} - \xi H_{\alpha_i}, \\
[\psi_i H_{\alpha_j}, \xi_{-i} X_{\alpha_i}] & = (\alpha_i,\alpha_j) \psi_i \xi_{-i} X_{\alpha_i} - \xi H_{\alpha_j}, \\
[\psi_j H_{\alpha_j}, \xi_{-j} X_{\alpha_j}] & = 2\psi_j\xi_{-j} X_{\alpha_j} - \xi H_{\alpha_j}.
\end{aligned}
\end{equation*}
We then obtain that
\begin{multline*}
(\alpha_i, \alpha_j)(\alpha_j, \alpha_i) [\psi_i H_{\alpha_i}, \xi_{-i} X_{\alpha_i}] - 4 [\psi_j H_{\alpha_i}, \xi_{-j} X_{\alpha_j}] - 2(\alpha_j,\alpha_i) [\psi_i H_{\alpha_j}, \xi_{-i} X_{\alpha_i}] \\ + 2(\alpha_j,\alpha_i)[\psi_j H_{\alpha_j}, \xi_{-j} X_{\alpha_j}] = (4 - (\alpha_i, \alpha_j)(\alpha_j, \alpha_i)) \xi H_{\alpha_i} \in \mathcal{L}_n.
\end{multline*}
Since $\alpha_i$ and $\alpha_j$ are distinct positive roots, $4 - (\alpha_i, \alpha_j)(\alpha_j, \alpha_i)  > 0$ and hence, $\xi H_{\alpha_i}\in \mathcal{L}_n$.

\xc{Sub-case 2.2: All the $\alpha_i$'s are the same.} Let $\alpha:= \alpha_i$ be the common root.
Note that
 $$X_\alpha \phi_i = \frac{1}{4}m_i (m_i + 2) \theta_i \quad \mbox{and} \quad
 X_{-\alpha} \phi_i = \frac{1}{4}m_i (m_i + 2)  \psi_i.$$
 Then,
\begin{equation*}
\begin{aligned}\relax
[\phi_1 X_\alpha, \xi_{-1} X_{-\alpha}]  & = \xi H_\alpha + \frac{1}{4}\sum_{i = 2}^n m_i(m_i + 2) \theta_i \xi_{-i} X_{-\alpha} - \frac{1}{4} m_1(m_1 + 2) \psi_1 \xi_{-1} X_\alpha, \\
[\theta_i H_{\alpha}, \xi_{-i} X_{-\alpha}] & = -2\theta_i \xi_{-i} X_{-\alpha} - \xi H_{\alpha},\\
[\psi_i H_{\alpha}, \xi_{-i} X_{\alpha}] & = 2\psi_i \xi_{-i} X_{\alpha} - \xi H_{\alpha}.
\end{aligned}
\end{equation*}
It follows that
\begin{multline*}
2 [\phi_1 X_\alpha, \xi_{-1} X_{-\alpha}]
+ \frac{1}{4}\sum_{i = 2}^n m_i(m_i + 2) [\theta_i H_\alpha, \xi_{-i} X_{-\alpha}]
 + \frac{1}{4}m_1(m_1 + 2) [\psi_1 H_\alpha, \xi_{-1} X_\alpha] \\
 = \left (2 -  \frac{1}{4}\sum_{i = 1}^n m_i(m_i + 2)\right )\xi H_\alpha \in \mathcal{L}_n.
\end{multline*}
Since $n \geq 2$ and since each $m_i$ is a positive, even integer,  $\frac{1}{4}\sum_{i = 1}^n m_i(m_i + 2) \geq 2n > 2$ and hence, $\xi H_\alpha\in \mathcal{L}_n$.

Now, let $H$ be an arbitrary element of $\frakh$. Decompose $H = H_1 + H_2$, where $H_1:= \frac{1}{2}\alpha_1(H) H_{\alpha_1}$ and $H_2  := H - \frac{1}{2}\alpha_1(H) H_{\alpha_1}$.
It suffices to show that $\xi H_2\in \mathcal{L}_n$.
Note that $\alpha_1(H_2) = 0$ and $H_2 \xi_{-1} = 0$.
Thus,
$$
[\psi_1 H_2,\xi_{-1}X_{\alpha_1}]
=  \psi_1\xi_{-1}\alpha_1(H_2) X_{\alpha_1}  + (H_2\xi_{-1})\psi_1 X_{\alpha_1} - (X_{\alpha_1}\psi_1)\xi_{-1}H_2 = -\xi H_2,
$$
which completes the proof.
\end{proof}

We have thus shown that $\xi H\in \mathcal{L}_n$ for all $H\in \frakh$.
To complete the proof of Theorem~\ref{thm:liealgebra}, it remains to establish the following result:

\begin{proposition}
    For any $\alpha\in \Delta$, $\xi X_\alpha\in \mathcal{L}_n$.
\end{proposition}

\begin{proof}
First, consider the case where there is an index $i\in \{1,\ldots, n\}$ such that $\lambda_i \neq -\alpha$.
Let $H\in \frakh$ be such that $(\lambda_i + \alpha)(H)\neq 0$. Then,
$$
[\phi_i X_\alpha, \xi_{-i}H] = -(\alpha + \lambda_i)(H) \xi X_\alpha  + \sum_{j = 1, j\neq i}^n  (X_\alpha \phi_j) \xi_{-j} H.
$$
Since $(X_\alpha \phi_j) \xi_{-j} H\in \mathcal{L}_n$, we have that $\xi X_\alpha\in \mathcal{L}_n$.

Now, consider the case where $\lambda_i = -\alpha$ for all $i = 1,\ldots, n$.
We appeal to Lemma~\ref{lem:sl2rep} to obtain an element $\psi_1\in \Phi_0$ such that $X_{-\alpha} \psi_1 = \phi_1$ (since $\phi_1 \in \Phi_{-\alpha}$, the decomposition $\phi_1 = \phi'_1 + \phi''_1$ is such that $\phi''_1 = 0$).
For each $i = 2,\ldots, n$, let $\eta_{-i}:= \psi_1 \prod_{j = 2,j\neq i}^n \phi_j$. Then,
\begin{equation}\label{eq:laststep}
[\psi_1 X_{\alpha}, \xi_{-1}  X_{-\alpha}] = \psi_1 \xi_{-1} H_\alpha + \sum_{i = 2}^n (X_\alpha \phi_i) \eta_{-i} X_{-\alpha} - \xi X_\alpha \in \mathcal{L}_n.
\end{equation}
The first term on the right hand side of~\eqref{eq:laststep} belongs to $\mathcal{L}_n$. For the second term, note that $X_\alpha \phi_i\in \Phi_0$ and $-\alpha \neq 0$, so the arguments in the first paragraph of the proof imply that $(X_\alpha \phi_i) \eta_{-i} X_{-\alpha}\in \mathcal{L}_n$. We thus conclude that $\xi X_\alpha\in \mathcal{L}_n$.
\end{proof}

\section{The three complete generators}
\label{sec:shared}

We now return to the setting where $G$ is a connected, semi-simple matrix group over $\F$.
Note that every vector field $V$ in $\frakg_L \oplus \frakg_R$ is algebraic and complete; indeed, for $V= L_X + R_Y$, we have that $e^{tV}(x) = e^{Yt}xe^{Xt}$ for any $x\in G$ and any $t\in \F$.
Also, for any $X,Y\in\frakg$,
$$
[L_X,L_Y]=L_{[X,Y]},\quad [R_X,R_Y]=-R_{[X,Y]},\quad\mbox{and}\quad [L_X,R_Y]=0.
$$
Since $\frakg_L \oplus \frakg_R \approx \frakg \oplus \frakg$ is a semi-simple Lie algebra  (where $L_X + R_Y \mapsto (X,-Y)$ is a Lie algebra isomorphism), the following lemma is an immediate consequence of~\cite[Theorem~6]{Kuranishi1951}:

\begin{lemma}\label{lem:firsttwo}
There exist $W_1, W_2\in \mathfrak g_L\oplus \mathfrak g_R$ such that $\operatorname{Lie}\{W_1,W_2\} = \mathfrak g_L\oplus\mathfrak g_R$. 
\end{lemma}

Explicit forms of $W_1$ and $W_2$ can be found in~\cite{CofDReal,CofDComplex}.
For the remainder of the section, we will exhibit the third complete vector field $W_3$ and show that $\Phi \mathfrak g_L \subseteq \operatorname{Lie}\{W_1,W_2,W_3\}$. We deal with the real- and complex-case in Subsections~\ref{ssec:realcase} and~\ref{ssec:complexcase}, respectively. 

We record below technical results that will be used in both cases.  
Let $\mathcal{A}(G)$ be the subspace of $\operatorname{End}(\mathbb{V})$ spanned by the elements of $G$.
Note that $\mathcal{A}(G)$ is closed under multiplication; indeed, for any two elements $u := \sum_{i} a_i x_i$ and $v := \sum_{j} b_j y_j$, with $a_i, b_j\in \F$ and $x_i,y_j\in G$, $$uv= \sum_{i,j} a_ib_j (x_iy_j) \in \mathcal{A}(G).$$
It follows that $\mathcal{A}(G)$ is a unital subalgebra of $\operatorname{End}(\mathbb{V})$.
Next, let $\mathcal{A}(\frakg)\subseteq \operatorname{End}(\mathbb{V})$ be the unital algebra generated by the elements of $\frakg$, known as the matrix representation of the universal enveloping algebra  of $\frakg$. We need the following lemma:

\begin{lemma}\label{lem:representation-algebra}
The two algebras $\mathcal{A}(G)$ and $\mathcal{A}(\frakg)$ are identical.
\end{lemma}

\begin{proof}
For any $X\in\mathfrak g$, $\exp(tX)\in \mathcal{A}(G)$ for any $t\in \R$.
Since $\mathcal{A}(G)$ is closed, we obtain that $$X = \lim_{t\to 0} \frac{1}{t}(\exp(tX) - I)\in \mathcal{A}(G).$$ Since $\mathcal{A}(G)$ is a unital algebra, $\mathcal{A}(\frakg)\subseteq \mathcal{A}(G)$.
Conversely, $\mathcal{A}(\frakg)$ is a linear subspace of $\operatorname{End}(\mathbb{V})$ and hence, is closed. Thus, for every $X\in\mathfrak g$, the infinite series $\exp(X) = \sum_{k= 0}^\infty \frac{1}{k!} X^k$ is an element of~$\mathcal{A}(\frakg)$.
The subgroup of $G$ generated by $\exp(X)$, for all $X\in \frakg$, contains an open neighborhood of the identity and hence, is an open subgroup of $G$. Since $G$ is connected, this subgroup coincides with~$G$.  It follows that $G\subseteq \mathcal{A}(\frakg)$. Since $\mathcal{A}(\frakg)$ is a vector space, we conclude that $\mathcal{A}(G)\subseteq \mathcal{A}(\frakg)$.
\end{proof}

We next consider the trace map $\tr: \mathcal{A}(\frakg)\to \F$ and the associated bilinear trace form $\tr(AB)$ for $A, B\in \mathcal{A}(\frakg)$. The {\it kernel} of the trace form is $\{A \in \mathcal{A}(\frakg)\mid \tr(AB) = 0 \mbox{ for any } B\in \mathcal{A}(\frakg) \}$. The trace form is said to be {\it nondegenerate} if the kernel is trivial.
We have the following result:

\begin{lemma}\label{lem:trace-pairing}
The trace form $\tr:\mathcal{A}(\frakg)\times \mathcal{A}(\frakg)\to\F$  is nondegenerate.
\end{lemma}

\begin{proof}
For the case $\F=\R$, we complexify $\mathcal{A}(\frakg)$,  extend the trace form bilinearly to $\mathcal{A}(\frakg)^\C\times \mathcal{A}(\frakg)^\C$, and show that   $\tr:\mathcal{A}(\frakg)^\C\times \mathcal{A}(\frakg)^\C\to\C$ is nondegenerate. 
For ease of notation, we will omit their super-index for the rest of the proof. All spaces below are over~$\C$.
Decompose $\mathbb{V} \approx \oplus_{\ell} m_\ell \mathbb{V}_{\ell}$, where the $\mathbb{V}_{\ell}$'s are pairwise nonisomorphic simple $\mathcal{A}(\frakg)$-modules and the $m_\ell$'s are their multiplicities. 
It follows that for any $A\in \mathcal{A}(\frakg)$,
$\tr(A) = \sum_{\ell} m_\ell \tr(A |_{\mathbb{V}_\ell})$.
Let $\operatorname{Ann}(\mathbb{V}_\ell)$ be the annihilator of $\mathbb{V}_\ell$, i.e.,
$$
\operatorname{Ann}(\mathbb{V}_\ell) := \left\{ A\in \mathcal{A}(\frakg) \mid A v=0 \mbox{ for any } v\in \mathbb{V}_\ell \right\}.
$$
Then, by~\cite[Lemma~3.2(b)]{BrownYakimov}, the kernel of the trace form coincides with $\cap_\ell\operatorname{Ann}(\mathbb{V}_\ell) = \operatorname{Ann}(\mathbb{V})$. Finally, since
$\mathcal{A}(\frakg)$ is a subspace of $\operatorname{End}(\mathbb{V})$, we conclude that $\operatorname{Ann}(\mathbb{V}) = \{0\}$.
\end{proof}

In the sequel, we use $\phi_A\in \Phi$, for $A\in\operatorname{End}(\mathbb{V})$, to denote the function $\phi_A(x):=\tr(xA)$. It is clear that $\Phi = \{\phi_A \mid A\in\operatorname{End}(\mathbb{V})\}$. Also, note that  
$L_X\phi_A=\phi_{XA}$ and $R_X\phi_A=\phi_{AX}$. We further have the following lemma: 

\begin{lemma}\label{lem:trace-coefficient-cyclic}
The map $A\mapsto\phi_A$ is a linear isomorphism from $\mathcal{A}(\frakg)$ to $\Phi$.
In particular, $\phi_I$ is a {\it cyclic vector} of $\Phi$ as a $\mathcal{U}(\frakg_R)$-module, i.e., $\mathcal{U}(\frakg_R)\phi_I=\Phi$.
\end{lemma}

\begin{proof}
We first show that the map is injective. 
Let $A \in \mathcal{A}(\frakg)$ be such that $\phi_A = 0$. 
Since $\mathcal{A}(G)$ is spanned by elements of $G$, $\phi_A(u) = 0$ for all $u\in \mathcal{A}(G)$. 
By Lemma~\ref{lem:representation-algebra}, $\mathcal{A}(G) = \mathcal{A}(\frakg)$, so $\tr(AB) = 0$ for all $B\in \mathcal{A}(\frakg)$. Then, by Lemma~\ref{lem:trace-pairing}, $A = 0$. 
We next show that the map is surjective. Note that every element $\phi\in\Phi$ can be obtained by restricting a linear functional $\tilde \phi$ on $\mathcal{A}(G)$ to $G$. 
Using again Lemmas~\ref{lem:representation-algebra} and~\ref{lem:trace-pairing}, there must exist an $A\in \mathcal{A}(\frakg)$ such that $\tilde \phi(B) = \tr(AB)$ for all $B\in \mathcal{A}(\frakg)$. It then follows that $\phi = \phi_A$. 
It is then an immediate consequence that $\mathcal{U}(\frakg_R)\phi_I=\{\phi_B\mid B\in\mathcal{A}(\frakg)\}=\Phi$.
\end{proof}

\subsection{The third generator for the real case}\label{ssec:realcase}

Throughout this subsection, we let $\F = \R$ and $G$ be a connected, semi-simple matrix group over~$\R$.  
Decompose $\mathfrak g=\oplus_{\ell = 1}^k \mathfrak g_\ell$, where each $\frakg_\ell$ is a simple ideal of~$\frakg$. For each $\ell$, let
$\mathfrak g_\ell=\mathfrak k_\ell\oplus\mathfrak p_\ell$ be a Cartan decomposition, with $\mathfrak{k}_\ell$ a maximally compact subalgebra of $\frakg_\ell$.
Let $Z_\ell\in \mathfrak{k}_\ell
$ be a nonzero element, and
$Z:= \sum_{\ell = 1}^k Z_\ell$.
We then define
\begin{equation}\label{eq:defineW3real}
W_3(x):= \phi_I(x) L_Z(x) \in \xalg(G).
\end{equation}
We have the following fact:

\begin{lemma}\label{lem:real-complete}
    The vector field $W_3$ is $\R$-complete.
\end{lemma}

\begin{proof}
Let $G = K\exp{\mathfrak{p}}$ be the global Cartan decomposition of $G$, where $\mathfrak{p} := \oplus_{\ell = 1}^k \mathfrak{p}_\ell$ and $K$ is the compact subgroup of $G$ whose Lie algebra is $\mathfrak{k} := \oplus_{\ell = 1}^k \mathfrak{k}_\ell$. For any $x\in G$, assuming that $e^{s W_3}(x)$ exists for all $s\in [0,t]$, we have that
$$e^{t W_3}(x) = x\exp\left (\int_0^t \phi_I(e^{s W_3}(x)) ds  Z \right ).$$
Since $Z\in \mathfrak{k}$, $e^{t W_3}(x)$ is in the compact set~$xK$.
It follows from the Escape Lemma~\cite[Lemma~9.19]{GTM218} that the maximal integral curve $t \mapsto e^{t W_3}(x)$ must be well defined for all $t\in \R$ and for all $x\in G$.
\end{proof}

The main result of the subsection is the following theorem:

\begin{theorem}
\label{thm:degreeone}
It holds that
$\Phi\mathfrak g_L\subseteq\operatorname{Lie}\{W_1,W_2,W_3\}.$
\end{theorem}

\begin{proof}
Let $S:=\mathcal{U}(\frakg_L \oplus \frakg_R) W_3$ be the subspace of $\xalg(G)$ spanned by $W_3$ and all the right-normed Lie products $[V_n, \cdots, [V_1, W_3]]$ for $V_i\in \frakg_L \oplus \frakg_R$.
By Lemma~\ref{lem:firsttwo}, it suffices to show that $\Phi\frakg_L \subseteq S$.
Consider the set
$$\mathfrak{i} := \left\{X\in\frakg\mid \Phi L_X\subseteq S \right\}.$$
It is clear that $\mathfrak{i}$ is a linear subspace of $\frakg$. It is also nontrivial because $Z$ belongs to $\mathfrak{i}$; indeed, since $[\frakg_L, \frakg_R] = 0$ and by Lemma~\ref{lem:trace-coefficient-cyclic},
$$\mathcal{U}(\frakg_R)W_3 = (\mathcal{U}(\frakg_R)\phi_I ) L_Z = \Phi L_Z.$$
Furthermore, for any $X\in\mathfrak{i}$, any $Y\in\mathfrak{g}$, and any $\phi\in\Phi$,
\begin{equation*}
[L_Y, \phi L_X]=(L_Y \phi)L_X + \phi L_{[Y,X]}.
\end{equation*}
Because $S$ is closed under $\frakg_L$,  $[L_Y, \phi L_X]\in S$.
Also, since $L_Y\phi\in\Phi$ and $X\in\mathfrak{i}$,
$(L_Y\phi)L_X\in  S$. Since $S$ is a vector space, we obtain that $\phi L_{[Y,X]}\in S$, which holds for any $\phi\in\Phi$. It follows that $[Y,X]\in\mathfrak{i}$ for any $Y\in \frakg$ and any $X\in \mathfrak{i}$, so $\mathfrak{i}$ is an ideal of $\mathfrak g$.
Every ideal of $\mathfrak g=\oplus_{\ell = 1}^k \mathfrak{g}_\ell$ is a direct sum of a subfamily of the simple ideals $\mathfrak g_\ell$. 
Since $Z = \sum_{\ell = 1}^k Z_\ell\in\mathfrak{i}$ 
and since each $Z_\ell$ is nonzero, it holds that $\mathfrak{i} = \mathfrak{g}$. We conclude that $\Phi\mathfrak g_L \subseteq S\subseteq\operatorname{Lie}\{W_1,W_2,W_3\}$.
\end{proof}

\subsection{\texorpdfstring{The third generator for the complex case}{The complex case}}
\label{ssec:complexcase}

Throughout this subsection, let $\F=\C$ and $G$ be a connected, semi-simple matrix group over~$\C$. 
It is known that $G$ admits a unique structure of an algebraic group~\cite[Chapter~4, Section~3.4, Corollary to Theorem~6]{OnishchikV} 
and that any holomorphic homomorphism from $G$ to $\mathrm{GL}(\mathbb{V}')$, for some space $\mathbb{V}'$, is polynomial~\cite[Chapter~3, Section~3.4, Theorem~4]{OnishchikV}. 
This fact allows us to work with any faithful, holomorphic representation of~$G$. 
Our choice of the representation~$\mathbb{V}$ follows the standard construction outlined in~\cite[Chapter~4, Section~3.4, Problems~44 and~45]{OnishchikV}. We provide below necessary detail for completeness of presentation.

Recall that $\frakg=\oplus_{\ell=1}^k\frakg_\ell$, where each $\frakg_\ell$ is a simple ideal of~$\frakg$. With slight abuse of notation, we also use $\frakg=\bigoplus_{\ell=1}^k\frakg_\ell$ to denote the adjoint representation of~$G$, with the $\frakg_\ell$'s its irreducible summands. 
Next, let $\tilde G$ be the simply connected covering group of~$G$. 
The center $Z(\tilde G)$ comprises all elements $z = \exp(2\pi \mathrm{i} H)$, for $H\in \frakh$, such that $\alpha(H) \in \mathbb{Z}$ for all $\alpha\in \Delta$.  
The character group of $Z(\tilde G)$ is isomorphic to $\Lambda_W/\Lambda_R$. 
Specifically, for each $\lambda\in \Lambda_W$, the associated character $\chi_\lambda$ is given by $\chi_\lambda(z) = e^{2\pi\mathrm{i}\lambda(H)}$. It is known (see, e.g.,~\cite[Chapter~VIII, Section~7.3, Proposition~8]{BourbakiLie79}) that each class $[\lambda]\in \Lambda_W/\Lambda_R$ contains a unique minuscule weight.

We write $\tilde G = \tilde G_1\times \cdots \times \tilde G_k$, where each $\tilde G_\ell$ is the simply connected Lie group with~$\frakg_\ell$ its Lie algebra. 
The above arguments also apply to the center $Z(\tilde G_\ell)$. For a given $\lambda\in \Lambda_W/\Lambda_R$, we write $\lambda = (\lambda_1,\ldots, \lambda_k)$, with each $\lambda_\ell\in \Lambda_{W_\ell}/\Lambda_{R_\ell}$ a minuscule weight.  

Let $K\subseteq Z(\tilde G)$ be the kernel of the covering map $\tilde G\to G$, so $Z(G) = Z(\tilde G)/K$. 
We assign to each nontrivial character $\chi_\lambda$ of $Z(G)$ a minuscule weight $\lambda$, and let $\Lambda^*$ be the collection of all such weights. 
For each $\lambda = (\lambda_1,\ldots,\lambda_k)\in \Lambda^*$, let $\mathbb{V}^{\lambda_\ell}$ be the representation of $\frakg_\ell$ with $\lambda_\ell$ the highest weight; such a representation exists by the highest weight theorem. 
Let $\mathbb{V}^{\lambda}:= \mathbb{V}^{\lambda_1}\otimes \cdots \otimes \mathbb{V}^{\lambda_k}$ be the irreducible representation of $\frakg$ with $\lambda$ the highest weight, obtained by the external tensor product. 
Since $\tilde G$ is simply connected, this Lie algebra representation can be lifted 
to 
the unique holomorphic representation of~$\tilde G$ whose differential at the identity~$e$ coincides with the Lie algebra representation 
(see, e.g.,~\cite[Chapter~III, Section~6.1, Theorem~1]{BourbakiLie13}). 
By Schur's lemma, each element~$z=(z_1,\ldots, z_k)$ in $Z(\tilde G)$ acts on $\mathbb{V}^\lambda$ as the scalar~$\chi_\lambda(z) = \prod_{\ell = 1}^k \chi_{\lambda_\ell}(z_\ell)$. In particular, $K$ acts trivially, so $\mathbb{V}^\lambda$ descends to a representation of $G = \tilde G/K$.

We now let $\mathbb{V}$ be the representation of $G$ as the direct sum of the adjoint representation~$\frakg$ and the $\mathbb{V}^\lambda$'s: 
\begin{equation}\label{eq:defV}
\mathbb{V}:=\bigoplus_{\ell=1}^k \frakg_\ell  \oplus \bigoplus_{\lambda\in\Lambda^*}\mathbb V^\lambda.
\end{equation}
By construction, the summands are nontrivial, irreducible, and pairwise nonisomorphic. It follows from~\cite[Chapter~VIII, Section~5.3, Corollary~1 to Proposition~4]{BourbakiAlgebra8} that
\begin{equation}\label{eq:decomp}
\mathcal A(\frakg)=\bigoplus_{\ell=1}^k\operatorname{End}(\frakg_\ell) \oplus
\bigoplus_{\lambda\in \Lambda^*}\operatorname{End}(\mathbb V^\lambda). 
\end{equation}
We have the following fact:

\begin{lemma}\label{lem:faithful}
The representation $\mathbb{V}$ of $G$ is faithful.
\end{lemma}

\begin{proof}
Since $G$ is connected, the kernel of its adjoint representation is $Z(G)$. For each $z\in Z(G)$ with $z \neq e$, there exists a $\lambda \in \Lambda^*$ such that $\chi_\lambda(z) \neq 1$, so $z$ acts nontrivially on $\mathbb V^\lambda$. It follows that the representation $\mathbb{V}$ is faithful.
\end{proof}

We now construct the third vector field $W_3$.  
Let $\mu_\ell$ be the highest root of $\frakg_\ell$,  
$Q_\ell$ be a nonzero element of the root space $\frakg_{\mu_\ell}$, and \begin{equation}\label{eq:defQ}
Q:=\sum_{\ell=1}^kQ_\ell \in \frakg.
\end{equation}
Since each summand of $\mathbb{V}$ in~\eqref{eq:defV} is invariant under~$Q$, the kernel of $Q^2$ is the direct sum of the kernels of $Q^2$ restricted to the summands. Then, by~\eqref{eq:decomp}, there exists a projection operator $P\in \mathcal{A}(\frakg)$ such that 
$$\operatorname{im} P = \operatorname{ker} Q^2 =: \mathbb{U}.$$
Now, define
\begin{equation}\label{eq:defineW3complex}
W_3(x):=\phi_P(x)L_Q(x) \in \xalg(G).
\end{equation}
We have the following fact:

\begin{lemma}\label{lem:complete}
The vector field $W_3$ is $\C$-complete.
\end{lemma}

\begin{proof}
It is known~\cite[Proposition~2.3]{TV06} that if $V$ is a complete holomorphic vector field and $\phi$ is a holomorphic function, then $\phi V$ is complete if and only if $V^2\phi=0$. 
In our case, $L_Q$ is $\C$-complete and $\phi_P$ is holomorphic, so it suffices to show that $L_Q^2\phi_P=0$.
Since $\operatorname{im}P=\ker Q^2$, we have that $Q^2P=0$ and hence,
$$
(L_Q^2\phi_P)(x)=\tr(xQ^2P)=0,
$$
for any $x\in G$.
In fact, we can express the flow explicitly as
$$
e^{tW_3}(x)=x\exp\bigl(\tau(t,x)Q\bigr),
$$
where
$$
\tau(t,x)=
\begin{cases}
\displaystyle\frac{e^{t\tr(xQP)}-1}{\tr(xQP)}\tr(xP) &\mbox{if }\tr(xQP)\neq0,\\[5pt]
t\tr(xP) &\mbox{if }\tr(xQP)=0.
\end{cases}
$$
Note that $\tau$ is holomorphic in $(t,x)$ and is well defined for any $x\in G$ and any $t\in\C$.
\end{proof}

\begin{remark}\label{rem:character}
If there were an element $Z=\sum_{\ell=1}^kZ_\ell\in\frakg$, with nonzero $Z_\ell\in\frakg_\ell$, such that~$Z$ acts as a square-zero operator on $\mathbb{V}$, 
then $W_3 = \phi_I L_Z$ would be a $\C$-complete vector field. Moreover, the same arguments in~Subsection~\ref{ssec:realcase} could be used to show that $\Phi\frakg_L \subseteq \operatorname{Lie}\{W_1,W_2,W_3\}$. Unfortunately, such an element $Z$ does not exist for the choice of our representation~\eqref{eq:defV}.  
To wit, note that $Z$ acts on the summand~$\frakg_\ell$ as $\operatorname{ad}_{Z_\ell}$ and that $\operatorname{ad}_{Z_\ell}^2\neq0$ by the Jacobson-Morozov theorem~\cite[Chapter~VIII, Section~11.2, Proposition~2]{BourbakiLie79}. 
On the positive side, there may exist another faithful representation $\mathbb{V}'$ of~$G$ and a desired square-zero operator~$Z\in \frakg$. For example, let $G$ be one of the following simple Lie groups: $\mathrm{SL}_n(\C)$ with $n\geq2$, $\mathrm{SO}_n(\C)$ with $n\geq5$, and $\mathrm{Sp}_{2n}(\C)$ with $n\geq1$ and let $\mathbb{V}'$ be the standard representation of~$G$. It has been shown in~\cite{CofDComplex} that a highest root vector $Z$ squares to zero. In the Appendix (see Proposition~\ref{prop:square-zero}) of this paper, we characterize the connected, complex simple groups that admit faithful representations $\mathbb{V}'$ with desired elements~$Z$ that square to zero on $\mathbb{V}'$.
\end{remark}

The main result of the subsection is the following: 

\begin{theorem}\label{thm:degree-one}
It holds that $\Phi\frakg_L\subseteq\operatorname{Lie}\{W_1,W_2,W_3\}$.
\end{theorem}

We still use $S$ to denote the space $\mathcal U(\frakg_L\oplus\frakg_R)W_3$. We show below that $\Phi\frakg_L \subseteq S + \frakg_R$.  
Let $\Phi^\lambda$ and $\Phi^{\ell}$ be spaces of the matrix-coefficients for $\mathbb V^\lambda$ and $\frakg_\ell$, respectively. 
Then, by Lemma~\ref{lem:trace-coefficient-cyclic} and~\eqref{eq:decomp},
$$
\Phi=\bigoplus_{\ell=1}^k\Phi^{\ell}\oplus\bigoplus_{\lambda\in \Lambda^*}\Phi^\lambda.
$$
We show below in Proposition~\ref{prop:minuscule} that $\Phi^\lambda \frakg_L\subseteq S$ for all $\lambda\in \Lambda^*$ and in Proposition~\ref{prop:adjoint} that $\Phi^{\ell}\frakg_L\subseteq S + (\frakg_\ell)_R$ for all $\ell = 1,\ldots, k$. 

For a given subspace $\mathbb{U}'\subseteq \mathbb{V}$, let
$$
\Phi_{\mathbb{U}'}:= \{\phi_A \mid A\in \mathcal{A}(\frakg) \mbox{ and } \operatorname{im}(A)\subseteq \mathbb{U}'\}.
$$
Note, in particular, that if $\mathbb{U}'$ is a subspace of an irreducible summand $\mathbb{V}'$ of $\mathbb{V}$, then 
the space of $A\in \mathcal{A}(\frakg)$ with $\operatorname{im}(A)\subseteq \mathbb{U}'$ is spanned by $uv^*$ with $u\in \mathbb{U}'$ and $v^*\in \mathbb{V}'^*$ and hence, $\Phi_{\mathbb{U}'}$ is spanned by the elements $\phi_{uv^*}$.  
We start with the following lemma:

\begin{lemma}\label{lem:coefficients}
It holds that $\mathcal U(\frakg_R)\phi_P=\Phi_{\mathbb U}$. In particular, $\mathcal{U}(\frakg_R)W_3 = \Phi_{\mathbb U}L_Q\subseteq S$.
\end{lemma}

\begin{proof}
Note that $\mathcal{U}(\frakg_R)\phi_P = \{\phi_{PB} \mid B\in \mathcal{A}(\frakg)\}$. Since $P\in \mathcal{A}(\frakg)$ and since $A\mapsto \phi_A$ is a linear isomorphism from $\mathcal{A}(\frakg)$ to $\Phi$ (by Lemma~\ref{lem:trace-coefficient-cyclic}), $\mathcal U(\frakg_R)\phi_P=\{\phi_{A}\mid A\in P\mathcal A(\frakg)\}$. 
Since $P\in\mathcal A(\frakg)$ is a projection with $\operatorname{im}P=\mathbb U$, we have that $P\mathcal A(\frakg)=\{A\in\mathcal A(\frakg)\mid\operatorname{im}A\subseteq\mathbb U\}$ and hence, $\mathcal U(\frakg_R)\phi_P=\Phi_{\mathbb U}$.
Finally, since $[\frakg_L,\frakg_R]=0$, $\mathcal U(\frakg_R)W_3=(\mathcal U(\frakg_R)\phi_P)L_Q=\Phi_{\mathbb U}L_Q$, which is contained in~$S$.
\end{proof}

We next establish the following result:

\begin{proposition}\label{prop:minuscule}
For any  $\lambda \in \Lambda^*$, $\Phi^\lambda\frakg_L\subseteq S$.
\end{proposition}

\begin{proof}
We write $\lambda = (\lambda_1,\ldots,\lambda_k)$. 
Since $\lambda_j$ is minuscule, the root vector $Q_j\in \frakg_j$ squares to zero on $\mathbb{V}^{\lambda_j}$ (this is trivial if $\lambda_j=0$; otherwise, see, e.g.,~\cite[Chapter~VIII, Section~7.3, Proposition~7]{BourbakiLie79}).
Let 
$\mathbb{U}^{\lambda_j}:= \ker(Q_j|_{\mathbb{V}^{\lambda_j}})$, which is nonzero since $Q^2_j |_{\mathbb{V}^{\lambda_j}} = 0$. For each $\ell = 1,\ldots, k$, let
$$
\mathbb{U}^\lambda_\ell:= \mathbb{U}^{\lambda_1}\otimes\cdots\otimes \mathbb{U}^{\lambda_{\ell-1}}\otimes\mathbb V^{\lambda_\ell}
\otimes \mathbb{U}^{\lambda_{\ell+1}}\otimes\cdots\otimes \mathbb{U}^{\lambda_k}.
$$
Note that $Q^2 u = 0$ for all $u\in \mathbb{U}^\lambda_\ell$, so $\mathbb{U}^\lambda_\ell \subseteq \mathbb{U}$. Thus, by Lemma~\ref{lem:coefficients}, 
$$
\Phi_{\mathbb{U}^\lambda_\ell}L_Q\subseteq \Phi_{\mathbb{U}} L_Q\subseteq S.
$$

Since $\mathbb{U}^\lambda_\ell$ is a subspace of the irreducible summand $\mathbb{V}^\lambda$ of $\mathbb{V}$, 
$\Phi_{\mathbb{U}^\lambda_\ell}$ is spanned by $\phi_{uv^*}$ for $u\in \mathbb{U}^\lambda_\ell$ and $v^*\in (\mathbb{V}^\lambda)^*$. 
Note that $L_Y\phi_{uv^*} = \phi_{(Yu)v^*}$. 
Since the $\ell$th tensor factor of $\mathbb U^\lambda_\ell$ is $\mathbb V^{\lambda_\ell}$, $Yu\in\mathbb U^\lambda_\ell$ for any $Y\in\frakg_\ell$ and any $u\in \mathbb{U}^\lambda_\ell$. 
It follows that for any $Y\in\frakg_\ell$ and any $\phi\in \Phi_{\mathbb{U}^\lambda_\ell}$, $(L_Y\phi)L_Q\in S$. 
Also, since $S$ is closed under $\frakg_L$, $[L_Y, \phi L_Q]\in S$. Further, note that $[Y, Q] =[Y,Q_\ell]$, so 
$$
\phi L_{[Y,Q_\ell]}=[L_Y,\phi L_Q]-(L_Y\phi)L_Q\in S. 
$$
Consider the set
$$
\mathfrak i_\ell:=\{X\in\frakg_\ell\mid \Phi_{\mathbb U^\lambda_\ell} L_X\subseteq S\},
$$
which is an ideal of $\frakg_\ell$ (one can obtain this fact by using the same arguments in the proof of Theorem~\ref{thm:degreeone}).
Moreover, $\mathfrak{i}_\ell$ is nonzero because it contains $[\frakg_\ell, Q_\ell]$. Thus, $\mathfrak{i}_\ell = \frakg_\ell$ and hence, 
$$
\Phi_{\mathbb{U}^\lambda_\ell} (\frakg_\ell)_L \subseteq S.
$$

For convenience, let $\frakg_{-\ell} := \oplus_{j = 1, j\neq\ell}^k \frakg_j$.  
Since $\mathbb{V}^{\lambda_j}$ is an irreducible representation of $\frakg_j$ and since $\mathbb{U}^{\lambda_j}$ is nonzero, $\mathcal{U}(\frakg_j)\mathbb{U}^{\lambda_j}= \mathbb{V}^{\lambda_j}$ which implies that $\mathcal{U}(\frakg_{-\ell}) \mathbb{U}^\lambda_\ell = \mathbb{V}^\lambda$ and hence,  
$\mathcal{U}((\frakg_{-\ell})_L)\Phi_{\mathbb{U}^\lambda_\ell} = \Phi^\lambda$. 
Since $[\frakg_{-\ell},\frakg_\ell] =0$, we have that $[L_Y,\phi L_X]=(L_Y\phi)L_X$ for any $Y\in\frakg_{-\ell}$, $X\in\frakg_\ell$, and $\phi\in\Phi$, and hence,
$$
\Phi^\lambda (\frakg_\ell)_L=(\mathcal{U}((\frakg_{-\ell})_L)\Phi_{\mathbb{U}^\lambda_\ell}) (\frakg_{\ell})_L=\mathcal{U}((\frakg_{-\ell})_L)(\Phi_{\mathbb{U}^\lambda_\ell} (\frakg_\ell)_L) \subseteq S,
$$
where the inclusion holds because $\Phi_{\mathbb{U}^\lambda_\ell} (\frakg_\ell)_L\subseteq S$ and $S$ is closed under $\frakg_L$. 
Since the above holds for all $\ell = 1,\ldots, k$, we thus conclude that $\Phi^\lambda \frakg_L \subseteq S$. 
\end{proof}

For the remainder of the subsection, we show that $\Phi^\ell \frakg_L \subseteq S + (\frakg_\ell)_R$. 
Let~$\mathfrak{s}$ be an arbitrary simple Lie algebra (over $\C$).  
Let~$B$ be the associated Killing form. We also use $B$ to denote the linear map $B:\mathfrak s\otimes\mathfrak s\to\C$ determined by the rule $B(X_1\otimes X_2):=B(X_1,X_2)$. 
Let $\{X_i\}$ and $\{X^i\}$ be bases of~$\mathfrak s$ that are dual to each other with respect to $B$, i.e., $B(X_i,X^j)=\delta_{ij}$, and $C_0:=\sum_iX_i\otimes X^i$, which is the Casimir element.
It is known that $B(C_0)=\dim\mathfrak s\neq0$, so $\mathfrak s\otimes\mathfrak s=\ker B\oplus\C C_0$.

We choose a Cartan subalgebra of $\mathfrak{s}$ and a positive root system. Let $\mu$ be the highest root of~$\mathfrak s$, $Q_\mu\in \mathfrak{s}_\mu$ be a nonzero root vector, and $\mathfrak u:=\ker\operatorname{ad}_{Q_\mu}^2$. 
Let $\mathfrak{s}$ act on $\mathfrak{s}\otimes \frak{s}$ 
as the tensor of two adjoint representations, i.e., $Y(X_1 \otimes X_2) = [Y,X_1]\otimes X_2 + X_1 \otimes [Y, X_2]$. 
We need the following lemma:

\begin{lemma}\label{lem:adjoint-tensor}
It holds that
$\mathcal U(\mathfrak s)(\mathfrak u\otimes Q_\mu)=\ker B$. 
\end{lemma}

\begin{proof}
Choose a root $\mathfrak{sl}_2$-triple $(Q_\mu,H_\mu,Q_{-\mu})$, with $[Q_\mu,Q_{-\mu}]=H_\mu$.
Note that $\mu$ is long (see, e.g.,~\cite[Section~10.4, Lemma~D]{HumphreysLie}).
It follows that for any root $\alpha$,  $\alpha(H_\mu) =(\alpha, \mu)$ can only take values $\pm 2$, $\pm 1$, and $0$. These  are the eigenvalues of $\operatorname{ad}_{H_\mu}$. Moreover, $\alpha(H_\mu) = 2$ (resp., $\alpha(H_\mu) = -2$) if and only if $\alpha = \mu$ (resp., $\alpha = -\mu$). 
Since each root space is one-dimensional, the eigenspaces of eigenvalues $-2$ and $2$ are $\C Q_{-\mu}$ and $\C Q_\mu$, respectively.
Let $\mathfrak s(j):=\{X\in\mathfrak s\mid \operatorname{ad}_{H_\mu}(X)=jX\}$ for $j\in\{\pm 2, \pm 1, 0\}$.
Since $\operatorname{ad}_{Q_\mu}^2$ maps $\mathfrak s(j)$ into $\mathfrak s(j+4)$, we obtain that $\mathfrak u=\bigoplus_{j=-1}^2\mathfrak s(j)$. Also, note that $\operatorname{ad}_{Q_\mu}^2 (Q_{-\mu}) = -2Q_\mu$. 
Thus, $\mathfrak s=\mathfrak u\oplus\C Q_{-\mu}$.
For any $X\in\mathfrak s(j)$, the 
$\operatorname{ad}(\mathfrak{s})$-invariance
of the Killing form gives
$$
0=B(\operatorname{ad}_{H_\mu}(X),Q_\mu)+B(X,\operatorname{ad}_{H_\mu}(Q_\mu))=(j+2)B(X,Q_\mu).
$$
Thus, $B(\mathfrak s(j),\mathfrak s(2))=0$ for all $j\neq-2$ and hence, $B(\mathfrak u,Q_\mu)=0$. Using again the 
$\operatorname{ad}(\mathfrak{s})$-invariance 
of $B$, we have that $M:=\mathcal U(\mathfrak s)(\mathfrak u\otimes Q_\mu)\subseteq\ker B$. 
Since $\mathfrak{s}$ is simple, $\mathcal{U}(\mathfrak{s})Q_\mu = \mathfrak{s}$ and hence, $\mathcal{U}(\mathfrak{s})(\mathfrak{s}\otimes Q_\mu) = \mathfrak{s}\otimes \mathfrak{s}$, so the quotient module $(\mathfrak{s}\otimes\mathfrak{s})/M$ is generated by the class $q:=[Q_{-\mu}\otimes Q_\mu]$. 
Note that $B(Q_{-\mu}, Q_\mu) = \frac{1}{2}B(H_\mu, H_\mu) \neq 0$, so $q\neq 0$. 

Now, let $X$ be an arbitrary positive root vector of $\mathfrak{s}$. 
Since $Q_\mu$ is a highest root vector, $[X,Q_\mu]=0$.
It follows that $\operatorname{ad}_X$ and $\operatorname{ad}_{Q_\mu}$ commute, which implies that
$$
\operatorname{ad}_{Q_\mu}^2([X,Q_{-\mu}])
=[X,\operatorname{ad}_{Q_\mu}^2(Q_{-\mu})]=-2[X,Q_\mu]=0,
$$
so $[X,Q_{-\mu}]\in\mathfrak u$. 
Note that $X(Q_{-\mu}\otimes Q_\mu)=[X,Q_{-\mu}]\otimes Q_\mu\in M$. Thus, $Xq=0$ in $(\mathfrak s\otimes\mathfrak s)/M$, so $q$ is the highest weight vector of the representation $(\mathfrak s\otimes\mathfrak s)/M$ of $\mathfrak{s}$.  
Further, for any element $H$ of the Cartan subalgebra of~$\mathfrak{s}$,
$H(Q_{-\mu}\otimes Q_\mu)=0$, 
so the representation $(\mathfrak{s}\otimes \mathfrak{s})/M$ of $\mathfrak{s}$ has zero as its highest weight (i.e., the representation is trivial). 
It follows that $(\mathfrak{s} \otimes \mathfrak{s})/M = \mathcal{U}(\mathfrak{s})q = \C q$. Thus, both $M$ and $\ker B$ have codimension one in $\mathfrak s\otimes\mathfrak s$.
Together with the inclusion $M\subseteq\ker B$, we conclude that $M=\ker B$.
\end{proof}

With the lemma above, we  establish the following result:

\begin{proposition}\label{prop:adjoint}
For any $\ell=1,\ldots,k$, $\Phi^\ell\frakg_L\subseteq S+(\frakg_\ell)_R$.
\end{proposition}

\begin{proof}
Recall that $G$ acts on the summand $\frakg_\ell$ of $\mathbb V$ by $\operatorname{Ad}$; in particular, $Q$ acts on $\frakg_\ell$ as $\operatorname{ad}_{Q_\ell}$, and $\frakg_j$ acts trivially on $\frakg_\ell$ for any $j\neq\ell$.
Let $\mathbb U^\ell:=\ker(\operatorname{ad}_{Q_\ell}^2|_{\frakg_\ell})$, which is nonzero, as it contains $Q_\ell$.
Since $\mathbb U^\ell=\frakg_\ell\cap\ker Q^2\subseteq\mathbb U$, by Lemma~\ref{lem:coefficients} we obtain that $\Phi_{\mathbb U^\ell}L_Q\subseteq\Phi_{\mathbb U}L_Q\subseteq S$.
For any $j\neq\ell$, the vector fields in $(\frakg_j)_L$ annihilate $\Phi_{\mathbb U^\ell}$, because $L_Y\phi_A=\phi_{YA}=0$ for any $Y\in\frakg_j$ and any $A\in\mathcal A(\frakg)$ with $\operatorname{im}A\subseteq\frakg_\ell$.
Thus, for any $Y\in\frakg_j$ and any $\phi\in\Phi_{\mathbb U^\ell}$,
$$
\phi L_{[Y,Q_j]}=[L_Y,\phi L_Q]-(L_Y\phi)L_Q=[L_Y,\phi L_Q]\in S.
$$
Using the same arguments in the proof of Proposition~\ref{prop:minuscule}, we have that $\mathfrak i_j:=\{X\in\frakg_j\mid\Phi_{\mathbb U^\ell}L_X\subseteq S\}$ is an ideal of $\frakg_j$ containing $[\frakg_j,Q_j]\neq0$. Thus, $\mathfrak i_j=\frakg_j$, so $\Phi_{\mathbb U^\ell}(\frakg_j)_L\subseteq S$ for all $j\neq\ell$.
Since $\phi L_{Q_\ell}=\phi L_Q-\sum_{j\neq\ell}\phi L_{Q_j}$ for any $\phi\in\Phi_{\mathbb U^\ell}$, we also obtain that $\Phi_{\mathbb U^\ell}L_{Q_\ell}\subseteq S$.

Since $\mathbb U^\ell\neq0$ and the adjoint representation of $\frakg_\ell$ is irreducible, $\mathcal U(\frakg_\ell)\mathbb U^\ell=\frakg_\ell$.
Note that $\Phi_{\mathbb U^\ell}$ (resp., $\Phi^\ell$) is spanned by the functions $\phi_{uv^*}$, for $u\in\mathbb U^\ell$ (resp., $u\in\frakg_\ell$) and $v^*\in\frakg_\ell^*$.  
Since $L_Y\phi_{uv^*}=\phi_{[Y,u]v^*}$ for any $Y\in\frakg_\ell$, we have that $\mathcal U((\frakg_\ell)_L)\Phi_{\mathbb U^\ell}=\Phi^\ell$.
For $j\neq\ell$ and for $Y\in\frakg_\ell$, $X\in\frakg_j$, and $\phi\in\Phi$, 
we have that $[L_Y,\phi L_X]=(L_Y\phi)L_X$. It follows that
$$
\Phi^\ell(\frakg_j)_L=\bigl(\mathcal U((\frakg_\ell)_L)\Phi_{\mathbb U^\ell}\bigr)(\frakg_j)_L=\mathcal U((\frakg_\ell)_L)\bigl(\Phi_{\mathbb U^\ell}(\frakg_j)_L\bigr)\subseteq S,
$$
where the inclusion holds because $\Phi_{\mathbb U^\ell}(\frakg_j)_L\subseteq S$ and $S$ is closed under $\frakg_L$.

It remains to show that $\Phi^\ell(\frakg_\ell)_L\subseteq S+(\frakg_\ell)_R$.
For $u,X\in\frakg_\ell$ and $v^*\in\frakg_\ell^*$, let
$
\Psi((u\otimes X)\otimes v^*):=\phi_{uv^*}L_X$, 
so 
$$
\Psi((u\otimes X)\otimes v^*)(x)=v^*(\operatorname{Ad}(x)u)L_X(x).
$$
It is clear that $\Psi$ is an isomorphism between $(\frakg_\ell\otimes\frakg_\ell)\otimes\frakg_\ell^*$ and $\Phi^\ell(\frakg_\ell)_L$.
For any $Y\in\frakg_\ell$, the identities $[L_Y,\phi L_X]=(L_Y\phi)L_X+\phi L_{[Y,X]}$ and $L_Y\phi_{uv^*}=\phi_{[Y,u]v^*}$ give rise to
$$
[L_Y,\Psi((u\otimes X)\otimes v^*)]
=\Psi\bigl(([Y,u]\otimes X+u\otimes[Y,X])\otimes v^*\bigr).
$$
Note that $\Phi_{\mathbb U^\ell}L_{Q_\ell}=\Psi((\mathbb U^\ell\otimes Q_\ell)\otimes\frakg_\ell^*)$.
Since $S$ is closed under $(\frakg_\ell)_L$, we obtain that
$$
\Psi\bigl((\mathcal U(\frakg_\ell)(\mathbb U^\ell\otimes Q_\ell))\otimes\frakg_\ell^*\bigr)\subseteq S.
$$
Then, by Lemma~\ref{lem:adjoint-tensor} with $\mathfrak s=\frakg_\ell$, $\Psi((\ker B)\otimes\frakg_\ell^*)\subseteq S$, where $B$ is the Killing form of $\frakg_\ell$.

Let $\{X_i\}$ and $\{X^i\}$ be two bases of $\frakg_\ell$, dual to each other with respect to $B$. 
Let $C_0:=\sum_iX_i\otimes X^i$ be the Casimir element. 
For any $Y\in\frakg_\ell$, let $v^*_Y:=B(Y,\cdot)$.
Since $B$ is $\operatorname{Ad}(G)$-invariant,
$$
\sum_iB(Y,\operatorname{Ad}(x)X_i)X^i=\sum_iB(\operatorname{Ad}(x^{-1})Y,X_i)X^i=\operatorname{Ad}(x^{-1})Y.
$$
Then, the value of $\Psi(C_0\otimes v^*_Y)$ at $x$ is $x\operatorname{Ad}(x^{-1})Y=Yx=R_Y(x)$.
Since $B$ is nondegenerate, every element of $\frakg_\ell^*$ has the form $v^*_Y$ for some $Y\in \frakg_\ell$.  
Using the fact that $\ker B \oplus \C C_0 = \frakg_\ell \otimes \frakg_\ell$ and the fact that $\Psi$ is an isomorphism between $(\frakg_\ell\otimes\frakg_\ell)\otimes\frakg_\ell^*$ and $\Phi^\ell(\frakg_\ell)_L$, 
we conclude that $\Phi^\ell(\frakg_\ell)_L\subseteq S+(\frakg_\ell)_R$.
\end{proof}

With the propositions above, we now prove Theorem~\ref{thm:degree-one}.

\begin{proof}[Proof of Theorem~\ref{thm:degree-one}]
By Propositions~\ref{prop:minuscule} and~\ref{prop:adjoint} and the decomposition of $\Phi$, we have that $\Phi\frakg_L\subseteq S+\frakg_R$.
By Lemma~\ref{lem:firsttwo}, $\frakg_L\oplus\frakg_R=\operatorname{Lie}\{W_1,W_2\}$, so $S=\mathcal U(\frakg_L\oplus\frakg_R)W_3\subseteq\operatorname{Lie}\{W_1,W_2,W_3\}$.
We conclude that $\Phi\frakg_L\subseteq S+\frakg_R\subseteq\operatorname{Lie}\{W_1,W_2,W_3\}$.
\end{proof}

\section{Proofs of the main results}\label{sec:proof-main-result}

Let $W_1, W_2\in \xalg(G)$ be given as in Lemma~\ref{lem:firsttwo}, and $W_3\in \xalg(G)$ be given in~\eqref{eq:defineW3real} for $\F=\R$, and in~\eqref{eq:defineW3complex} for $\F=\C$. 

\subsection{Proof of Theorem~\ref{thm:main}}
\label{sec:proof-main}

Let $\Phi''$ be the subspace of constant functions in~$\Phi$ (so that $\Phi''=0$ if $\F=\C$). Since $\Phi$ is a finite-dimensional $\frakg$-module and $\frakg$ is semi-simple, complete reducibility yields a $\frakg$-invariant decomposition $\Phi=\Phi'\oplus\Phi''$. Then, $\Phi'$ does not contain any nonzero constant function. Since $\mathcal{A}(\Phi')$ is unital and $\Phi''\subseteq\F\bfo$, we have $\mathcal{A}(\Phi')=\mathcal{A}(\Phi)$.
By Theorem~\ref{thm:liealgebra}, Theorem~\ref{thm:degreeone} (for $\F=\R$), and Theorem~\ref{thm:degree-one} (for $\F=\C$), $\mathcal{A}(\Phi')\frakg_L\subseteq \operatorname{Lie}\{W_1,W_2,W_3\}$.
Also, since each $W_i$ is algebraic, we have that $\operatorname{Lie}\{W_1,W_2,W_3\} \subseteq \xalg(G)$.
It now remains to show that $\mathcal{A}(\Phi)\frakg_L = \xalg(G)$. It is clear that any vector field in $\mathcal{A}(\Phi)\frakg_L$ is algebraic. Conversely,
given a $V\in \xalg(G)$, we show below that $V\in \mathcal{A}(\Phi)\frakg_L$.
Since $G$ is a connected, semi-simple matrix group, $\det(x) = 1$ for all $x\in G$. Thus, $x^{-1} = \operatorname{adj}(x)$. It follows that $x^{-1} V(x)\in \frakg$ is algebraic.
Let $X_1,\ldots, X_m$ be a basis of $\frakg$. Then, there exist coefficients $a_1,\ldots, a_m\in \mathcal{A}(\Phi)$ such that
$x^{-1} V(x) = \sum_{i = 1}^m a_i(x) X_i$,
which implies that $V = \sum_{i = 1}^m a_iL_{X_i}\in \mathcal{A}(\Phi)\frakg_L$.
\hfill{\qed}

\begin{remark}
If every irreducible summand of $\mathbb V$ is nontrivial, as is the case for the representation~\eqref{eq:defV}, 
then $\Phi$ contains no nonzero constant function, so $\Phi''=0$ and $\Phi'=\Phi$.
To wit, by Lemma~\ref{lem:trace-coefficient-cyclic}, every $\phi\in\Phi$ has the form $\phi_A$ for some $A\in\mathcal A(\frakg)$, and $A\mapsto\phi_A$ is injective.
If $\phi_A$ is constant, then $\phi_{XA}=L_X\phi_A=0$ and hence $XA=0$ for all $X\in\frakg$.
Thus, $\frakg$ acts trivially on $\operatorname{im}A$, which is therefore a trivial sub-representation of $\mathbb V$ and hence zero.
It follows that $A=0$, i.e., $\phi_A=0$.
\end{remark}

\subsection{Proof of Corollary~\ref{cor:B}}
\label{sec:consequences}

The proof, for either $\F = \R$ or $\F = \C$, builds upon~\cite[Theorem~2.5]{Varolin01}, which we recall below. 
Let $M$ be a smooth (resp., complex) manifold, $\mathfrak X(M)$ be the space of smooth (resp., holomorphic) vector fields on~$M$, and $\operatorname{Diff}(M)$ be the group of diffeomorphisms (resp., biholomorphisms) of~$M$.
For a family $\{V_\alpha\}_{\alpha\in \mathcal{I}}$ of $\R$-complete vector fields in $\mathfrak X(M)$, let $\mathcal{V}$ be the subgroup of $\operatorname{Diff}(M)$ generated by $e^{tV_\alpha}$, for $t\in\R$ and $\alpha\in \mathcal{I}$ and $\mathfrak{v}$ be the Lie algebra generated by the $\{V_\alpha\}_{\alpha\in \mathcal{I}}$. 
We denote by $\overline{\mathcal{V}}$ (resp., $\overline{\mathfrak{v}}$) the closure of $\mathcal{V}$ in $\operatorname{Diff}(M)$ (resp., $\mathfrak{v}$ in $\mathfrak{X}(M)$) with respect to the compact-open topology. 
For a $\mathrm{C}^1$ path $\{\Psi_t\}_{t\in[0,1]}$ in $\operatorname{Diff}(M)$, i.e., $(t,x)\mapsto\Psi_t(x)$ is of class $\mathrm{C}^1$, let
\begin{equation}\label{eq:defUt}
U_t:=\partial_t\Psi_t\circ\Psi_t^{-1}\quad \mbox{for } t\in[0,1].
\end{equation}

The following theorem adapted from~\cite[Theorem~2.5]{Varolin01} was stated for a complex manifold $M$.  
Its proof relies on the use of Gronwall's inequality, \cite[Proposition~2.4]{Varolin01}, and~\cite[Theorem~2.3]{Varolin01}, the last of which builds directly upon~\cite[Theorem~2.1.26]{AbrahamMarsden}  in the smooth setting. Thus, the arguments in the proof hold for smooth vector fields, so the same statement applies to smooth manifolds.

\begin{theorem}[cf.~{\cite[Theorem~2.5]{Varolin01}}]\label{thm:varolin}
If $\Psi_0=\operatorname{id}_M$ and $U_t\in\overline{\mathfrak{v}}$ for all $t\in[0,1]$, then $\Psi_1\in \overline{\mathcal{V}}$.
\end{theorem}

We present below arguments that enable the use of Theorem~\ref{thm:varolin} for complex and real cases, and complete the proof of Corollary~\ref{cor:B}.

\subsubsection{Proof for the complex case}\label{sec:B-complex}

Let $g\in\operatorname{Diff}_0(G)$, and $\Psi$ be a smooth isotopy in $\operatorname{Diff}(G)$ with $\Psi_0=\operatorname{id}_G$ and $\Psi_1=g$. Let $U_t$ be given as in~\eqref{eq:defUt}.  
Since $G\subseteq\operatorname{End}(\mathbb V)$, the maps $(\Psi_{t+h}-\Psi_t)/h:G\to\operatorname{End}(\mathbb V)$, for $h\neq0$ with $t+h\in[0,1]$, are holomorphic, and since $\Psi$ is smooth on $[0,1]\times G$, they converge to $\partial_t\Psi_t$ uniformly on compact sets as $h\to0$.
Thus, $\partial_t\Psi_t$ is holomorphic and hence, so is $U_t$.
Each vector field in the family $\{cW_i\mid c\in\C,\ i=1,2,3\}$ is $\R$-complete. It is clear that the subgroup $\mathcal{V}$ generated by $e^{tcW_i}$ is $\mathcal{W}$ and that the Lie algebra $\mathfrak{v}$ generated by this family is $\operatorname{Lie}\{W_1,W_2,W_3\}$, which is $\xalg(G)$ by Theorem~\ref{thm:main}. Note that $\xalg(G)$ is dense in $\mathfrak X(G)$ (see, e.g.,~\cite[proof of Theorem~3.4]{TV06}). 
Thus, $U_t\in \overline{\mathfrak{v}}= \mathfrak{X}(G)$ and hence, by Theorem~\ref{thm:varolin}, $g\in \overline{\mathcal{V}}=\overline{\mathcal W}$, which is the closure of $\mathcal W$ with respect to the compact-open topology (by Cauchy's estimate, the $\coinf$-topology as well).
\hfill{\qed}

\subsubsection{Proof for the real case}\label{sec:B-real}

We start with the following lemma:

\begin{lemma}\label{lem:smooth-Vfieldapproximate}
It holds that $\operatorname{Lie}\{W_1,W_2,W_3\}$ is dense in $\mathfrak X(G)$ with respect to the $\coinf$-topology. 
\end{lemma}

\begin{proof}
The set $\Phi \cup \{\bfo\}$ strongly separates points, i.e., $\bfo$ is nowhere zero and for any two distinct points $x, y\in G$, there exists a matrix coefficient $\phi\in \Phi$ such that $\phi(x)\neq \phi(y)$. Also, $\Phi$ separates tangent vectors. To wit, we fix a basis of $\mathbb{V}$ and let $c_{ij}(A)$ be the $ij$th entry of $A\in \operatorname{End}(\mathbb{V})$.
Given a nonzero $A\in T_xG$, say $c_{ij}(A)\neq 0$, the matrix coefficient $c_{ij}: G\to \R$  satisfies $d_x c_{ij}(A) = c_{ij}(A)$.
Thus, by Nachbin's Theorem~\cite[Theorem~1.2.1]{Llavona}, $\mathcal{A}(\Phi)$ is dense in $\mathrm{C}^\infty(G)$ with respect to the $\coinf$-topology.
Next, let $X_1,\ldots, X_m$ be a basis of $\frakg$, so $L_{X_1},\ldots, L_{X_m}$ is a global frame on $G$. Then, for any $V\in \mathfrak{X}(G)$, there exist $a_1,\ldots, a_m\in \mathrm{C}^\infty(G)$ such that $V = \sum_{i = 1}^m a_i L_{X_i}$.
It follows that
$\mathfrak{X}(G) = \operatorname{cl}_{\coinf}(\mathcal{A}(\Phi)\frakg_L) = \operatorname{cl}_{\coinf}(\operatorname{Lie}\{W_1,W_2,W_3\})$, where $\operatorname{cl}_{\coinf}(\cdot)$ denotes the closure of a set with respect to the $\coinf$-topology.
\end{proof}

Given a vector field $V\in \mathfrak{X}(G)$ and given~$k\in \N_0$, the {\it vertical prolongation} of $V$, which we denote by $j^kV$ (following the notation of~\cite{krupka2001}), is the smooth vector field on $J^k(G,G)$ such that
$$
j^kV(j_x^k f) = \left. \frac{d}{dt} \right |_{t = 0} j^k_x(e^{tV}\circ f),
$$
for any smooth map $f: G\to G$. It follows that
\begin{equation}\label{eq:exponentialofprolongV}
e^{t j^kV}(j_x^kf)=j_x^k(e^{tV}\circ f).
\end{equation}
In particular, if $V$ is complete, then so is $j^kV$.

Note that the map $j^k :\mathfrak{X}(G)\to \mathfrak{X}(J^k(G,G))$, defined by sending $V$ to $j^kV$, is linear and continuous, where $\mathfrak{X}(G)$ is equipped with the $\coinf$-topology 
and $\mathfrak{X}(J^k(G,G))$ is equipped with the compact-open topology (one can establish this fact by using the local coordinates of fiber charts~\cite[Section~3]{krupka2001}).
Also, note that $j^k$ is Lie-bracket preserving, i.e., $j^k[U, V] = [j^kU, j^kV]$ for any $U, V\in \mathfrak{X}(G)$. 
Let $\mathfrak{v}\subseteq \mathfrak{X}(J^k(G,G))$ be the Lie algebra generated by $j^kW_1$, $j^kW_2$, and $j^kW_3$. 
Then, 
$$\mathfrak{v}=\{j^k V  \mid V\in \operatorname{Lie}\{W_1,W_2,W_3\}\}.$$  
By Lemma~\ref{lem:smooth-Vfieldapproximate}, $\operatorname{Lie}\{W_1,W_2,W_3\}$ is dense in $\mathfrak X(G)$. Since~$j^k :\mathfrak{X}(G)\to \mathfrak{X}(J^k(G,G))$ is continuous, 
$$j^k(\mathfrak{X}(G))\subseteq \overline{\mathfrak{v}}.$$

For any $h\in\operatorname{Diff}(G)$, let $J^kh$ be the diffeomorphism of $J^k(G,G)$ defined by $J^kh(j^k_xf):=j^k_x(h\circ f)$.
Note that $J^k(h_2\circ h_1)=J^kh_2\circ J^kh_1$ for any $h_1,h_2\in\operatorname{Diff}(G)$. Also, by~\eqref{eq:exponentialofprolongV}, $e^{tj^kV}=J^k(e^{tV})$. 
Let~$\mathcal{V}$ be the subgroup of $\operatorname{Diff}(J^k(G,G))$  generated by $e^{t j^kW_i}$ for $t\in \R$ and $i = 1,2,3$. 
Then, 
$$
\mathcal{V} = \{J^k h\mid h\in\mathcal W\}.
$$

Now, let $g$ be an arbitrary element of $\operatorname{Diff}_0(G)$, $k\in\N_0$, and $\Psi$ be a smooth isotopy in $\operatorname{Diff}(G)$ with $\Psi_0=\operatorname{id}_G$ and $\Psi_1=g$. 
Let $U_t:=\partial_t\Psi_t\circ\Psi_t^{-1}$ be defined in the same way as in~\eqref{eq:defUt}. Since~$\Psi$ is smooth, $U_t\in \mathfrak{X}(G)$ for all $t\in [0,1]$. 
Then, $t\mapsto J^k\Psi_t$ is a smooth isotopy in $\operatorname{Diff}(J^k(G,G))$. We claim that
\begin{equation}\label{eq:prolongisotopy}
\partial_t J^k\Psi_t=(j^kU_t)\circ J^k\Psi_t \quad \mbox{for }t\in[0,1].
\end{equation} 
To wit, note that $J^k\Psi_{t+h}=J^k(\Psi_{t+h}\circ\Psi_t^{-1})\circ J^k\Psi_t$. The $k$-jets of $\Psi_{t+h}\circ\Psi_t^{-1}$ and of $e^{hU_t}$ at any point of $G$ agree up to the first order in $h$, both maps being $\operatorname{id}_G$ at $h=0$ with $h$-derivative $U_t$. It follows that $$\partial_hJ^k(\Psi_{t+h}\circ\Psi_t^{-1})|_{h=0}=\partial_hJ^k(e^{hU_t})|_{h=0}=j^kU_t,$$
which establishes~\eqref{eq:prolongisotopy}.

Finally, since $j^kU_t\in j^k(\mathfrak{X}(G))\subseteq\overline{\mathfrak v}$ for any $t\in[0,1]$, we appeal to Theorem~\ref{thm:varolin} to have that $J^kg\in \overline{\mathcal{V}}$, so
$g$ belongs to the closure of $\mathcal{W}$ with respect to the compact-open $\mathrm{C}^k$-topology. Since this holds for all $k\in \N_0$, we conclude that $g\in \operatorname{cl}_{\coinf}(\mathcal{W})$. 
\hfill{\qed}

\printbibliography

\section*{Appendix}\label{sec:examples}

In the appendix, we let $G$ be a connected, complex simple group, by which we mean a connected complex Lie group whose Lie algebra~$\frakg$ is simple. 
We retain the same notation used in Section~\ref{ssec:complexcase}. 
A group $G$ is said to be {\it adjoint} if its center is trivial.  
Since $\frakg$ is simple, the element~$Q$ introduced in~\eqref{eq:defQ} is now a root vector of the highest root~$\mu$.
We have the following result:

\begin{proposition}\label{prop:square-zero}
Let $G$ be a connected, complex simple group.
Then, there exist a faithful finite-dimensional holomorphic representation $\rho: G\to \operatorname{GL}(\mathbb{V})$ of $G$ and a nonzero $X\in\frakg$ such that $d\rho(X)^2=0$ if and only if $G$ is not an adjoint group of type $A$, $D$, or $E$.
\end{proposition}

\begin{proof}
Suppose first that such $\rho$ and $X$ exist; then, consider 
the set $$\mathcal{C}:=\{Y\in\frakg\mid d\rho(Y)^2=0\},$$
which contains the nonzero element~$X$. 
Since the entries of $d\rho(Y)^2$ are homogeneous quadratic polynomials in $Y$, $\mathcal C$ is Zariski closed in $\frakg$ and stable under scalar multiplication. 
Let $\mathbb{P}(\mathcal{C})$ be the set of lines in $\frakg$ contained in~$\mathcal C$, i.e., the projectivization of~$\mathcal C$, 
which is a nonempty closed subvariety of the projective space
$\mathbb P(\frakg)$ and hence a complete variety (see, e.g.,~\cite[Sections~1.6 and~6.2]{HumphreysLAG}).

Note that $\mathbb{P}(\mathcal{C})$ is $\operatorname{Ad}(G)$-invariant; indeed, $d\rho(\operatorname{Ad}(x)Y)^2=\rho(x)d\rho(Y)^2 \rho(x)^{-1}$ for any $Y\in \frakg$. 
Given the positive system $\Delta^+$, let $B$ be the Borel subgroup of $\operatorname{Ad}(G)$, which acts on $\mathbb{P}(\mathcal{C})$. 
Note that~$B$ is connected, solvable, and affine algebraic. 
By the Borel fixed point theorem~\cite[Section~21.2]{HumphreysLAG}, the $B$-action has a fixed point. 
Let $\C Y_0$ be the corresponding line in $\mathcal{C}$, stable under the action of~$B$.
It follows that $\operatorname{ad}_{X_\alpha}(Y_0)\in\C Y_0$ and hence, $\operatorname{ad}_{X_\alpha}(Y_0)=0$ for all positive root vectors $X_\alpha$, as $\operatorname{ad}_{X_\alpha}$ is nilpotent. Since $\frakg$ is simple, $Y_0\in\frakg_\mu=\C Q$, so $d\rho(Q)^2=0$. 
It follows that $d\rho(Y)^2=0$ for every $Y$ in the orbit~$\operatorname{Ad}(G)Q$. 
Since $Q$ is a highest root vector, we have that $[Q,[Q,\frakg]]\subseteq\C Q$ and hence, 
\begin{equation}\label{eq:extremalelements}
[Y,[Y,\frakg]]\subseteq\C Y, \quad \mbox{for any } Y\in \operatorname{Ad}(G)Q.
\end{equation} 
Elements $Y$ that satisfy~\eqref{eq:extremalelements} are called extremal elements~\cite{CSUW01}. In particular, these elements include,  for each long root $\alpha\in \Delta$ (with $\|\alpha\| = \|\mu\|$), a root vector in $\frakg_\alpha$. 

Let $\mathbb{V}'$ be an irreducible summand of $\mathbb{V}$ and $\rho': G\to \operatorname{GL}(\mathbb{V}')$ be the representation induced by~$\rho$. 
Let $\tilde \rho'$ be the representation of $\tilde G$ obtained by the pullback of $\rho'$ via the covering map $\tilde G\to G$. Note that $d\tilde \rho'_e = d\rho'_e$, which is a representation of~$\frakg$ on $\mathbb{V}'$.    
The elements of $\operatorname{Ad}(G)Q$ are extremal and square to zero on $\mathbb{V}'$. 
The span of these extremal elements is $\operatorname{Ad}(G)$-invariant and hence, is the entire $\frakg$ (since $\frakg$ is simple).  
It follows that $\mathbb{V}'$ is a {\it quadratic} $\frakg$-module (following the terminology of~\cite[Definition~7.1]{CSUW01}). 

If $\frakg$ is simply laced (i.e., $\frakg$ is of type $A$, $D$, or $E$), then by~\cite[Lemma~7.4]{CSUW01} 
the highest weight of $\mathbb{V}'$ is minuscule.
If, further, $G$ is adjoint, 
then the kernel~$K$ of the covering map $\tilde G\to G$ is $Z(\tilde G)$. 
Thus, $\tilde \rho'(z)$ acts trivially on $\mathbb V'$ for any $z\in Z(\tilde G)$ and hence, the highest weight of $\mathbb V'$ belongs to $\Lambda_R$.  
Since $\frakg$ is simply laced, it is known~\cite[Chapter~VIII, Section~7.3, Proposition~8]{BourbakiLie79} that $\Lambda_R$ does not contain any nonzero minuscule weight, so $\mathbb V'\approx \C$.
Thus, if $G$ is an adjoint group of type $A$, $D$, or $E$, then $\rho$ is trivial, contradicting the hypothesis that $\rho$ is faithful.

Conversely, suppose first that $G$ is not adjoint;
then, $Z(G)=Z(\tilde G)/K$ is nontrivial and hence,  $\Lambda^*\neq\varnothing$. Let $\rho$ be the representation $\bigoplus_{\lambda\in\Lambda^*}\mathbb V^\lambda$ of $G$. 
Since each $\lambda\in\Lambda^*$ is minuscule, every root vector  squares to zero on $\mathbb V^\lambda$. 
Since $Q$ is a highest root vector, $d\rho(Q)^2=0$. 
Also, since $\frakg$ is simple and $\rho$ is nontrivial, $d\rho_e$ is injective. It follows that $\ker\rho$ is a discrete normal subgroup of the connected group $G$, so $\ker\rho\subseteq Z(G)$. 
For any $z\in Z(G)$ with $z\neq e$, there exists a $\lambda\in\Lambda^*$ with $\chi_\lambda(z)\neq1$, and $z$ acts on $\mathbb V^\lambda$ as the scalar multiplication $\chi_\lambda(z)$. Thus, $z\notin\ker\rho$ and hence, $\rho$ is faithful.

Finally, we assume that $G$ is an adjoint group of type $B_n$, $C_n$, $F_4$, or $G_2$. 
Let $\mu_s$ be the highest short root, and $\tilde\rho$ be the irreducible representation $\mathbb{V}$ of $\tilde G$ with $\mu_s$ the highest weight.
Since $\mu_s\in\Lambda_R$, $\tilde \rho(z)$ acts as the identity map on $\mathbb{V}$ for any $z\in Z(\tilde G)$, so $\tilde \rho$ induces the representation $\rho$ of $G$ on $\mathbb{V}$.
By the same arguments as above, $\ker\rho\subseteq Z(G)$, which is trivial and hence, $\rho$ is faithful.
Because the weights of $\rho$ are in the convex hull of the Weyl group orbit of $\mu_s$, every such weight $\nu$ satisfies $\|\nu\|\leq\|\mu_s\|<\|\mu\|$.  
Since $d\rho(Q)^2$ maps the $\nu$-weight space into the $(\nu+2\mu)$-weight space and $\|\nu+2\mu\|\geq2\|\mu\|-\|\nu\|>\|\mu_s\|$, we conclude that $d\rho(Q)^2=0$.
\end{proof}

\end{document}